\documentclass[aop, preprint]{imsart}
\usepackage{amsthm,amsmath,amsfonts,amssymb}
\usepackage[numbers,sort&compress]{natbib}
\usepackage[colorlinks,citecolor=blue,urlcolor=blue]{hyperref}
\usepackage{graphicx}
\usepackage[british]{babel}
\usepackage[usenames,dvipsnames]{xcolor}
\usepackage{anyfontsize}
\usepackage{amssymb}
\usepackage{comment}
\usepackage{mathrsfs}
\usepackage[expansion=true, protrusion=true]{microtype}

\usepackage{subcaption}
\usepackage{longtable}
\usepackage[T1]{fontenc}

\usepackage{enumitem}

\usepackage{tabularx}
\usepackage{longtable}

\definecolor{brick}{HTML}{FF0800}

\makeatletter

\DeclareRobustCommand{\cev}[1]{%
\mathpalette\do@cev{#1}%
}
\newcommand{\do@cev}[2]{%
 \fix@cev{#1}{+}%
 \reflectbox{$\m@th#1\vec{\reflectbox{$\fix@cev{#1}{-}\m@th#1#2\fix@cev{#1}{+}$}}$}%
 \fix@cev{#1}{-}%
}
\newcommand{\fix@cev}[2]{%
 \ifx#1\displaystyle
  \mkern#23mu
 \else
  \ifx#1\textstyle
   \mkern#23mu
  \else
   \ifx#1\scriptstyle
    \mkern#22mu
   \else
    \mkern#22mu
   \fi
  \fi
 \fi
}

\makeatother

\theoremstyle{plain}
  \newtheorem{theorem}{Theorem}[section]
  \newtheorem{lemma}[theorem]{Lemma}
  
  \newtheorem{corollary}[theorem]{Corollary}
  
  \newtheorem{fact}[theorem]{Fact}

\theoremstyle{definition}
  
  \newtheorem{definition}[theorem]{Definition}
  
  \newtheorem{example}[theorem]{Example}

  \newtheorem{PA}{Principal Assumption}

\theoremstyle{remark}
  \newtheorem{remark}[theorem]{Remark}

  \numberwithin{equation}{section}

\def \N {\mathbb N}

\def \R {\mathbb R}

\newcommand*{\e}[1]{\text{e}^{#1}}
\begin{document}

\begin{frontmatter}

\title{Multitype branching processes in random environments with not strictly positive expectation matrices}
\runtitle{Multitype branching processes in random environments}

\begin{aug}
	\author[A]{\fnms{Vilma}~\snm{Orgov\'anyi}\ead[label=e1]{orgovanyi.vilma@gmail.com}}
	\and
	\author[A]{\fnms{K\'aroly}~\snm{Simon}\ead[label=e2]{karoly.simon51@gmail.com}}

	\address[A]{Department of Stochastics, Institute of Mathematics, Budapest University of Technology and Economics, M\H{u}egyetem rkp. 3., H-1111 Budapest, Hungary\printead[presep={,\ }]{e1,e2}}
	
	\end{aug}
	
\begin{abstract}
	It is well known that under some conditions the almost sure survival probability of a multitype branching processes in random environment is positive if the Lyapunov exponent corresponding to the expectation matrices is positive, and zero if the Lyapunov exponent is negative. The goal of this note is to establish similar results when certain positivity conditions on the expectation matrices are not met. One application of such a result is to classify the positivity of Lebesgue measure of certain overlapping random self-similar sets in the line.
\end{abstract}

	
	
	\end{frontmatter}

\section{Introduction}

In this paper we consider the extinction probability of multitype branching processes in random environments $\left\{\mathbf{Z}_n\right\}_{n=1}^{\infty}$, formally defined starting in Section \ref{w88}.
Our main theorem (Theorem \ref{z68}) states that, under mild conditions, the positivity of the Lyapunov exponent corresponding to the expectation matrices determines the positivity of the survival probability.

Informally, a branching process is:
\begin{itemize}
	\item \textit{multitype} if each individual has a type and the type determines the distribution according to which it gives birth to different types of individuals;
	\item \textit{temporally non-homogeneous} or \textit{in varying environment} if we allow the offspring distribution to change over time in a predefined deterministic manner; and
	\item \textit{in a random environment} if the temporal non-homogeneity is non-deterministic.
\end{itemize}

More precisely, let $N \geq 2$ (the number of types) and assume we are given a stationary ergodic sequence
$\left\{\theta _n\right\}_{n \geq 1}$ of random variables called the environmental sequence. For almost every realization we consider an associated
sequence of $N$-dimensional vector of $N$-variate probability generating functions (\textit{pgf}s)
$$\left\{\mathbf{f}_{\theta _n}(\mathbf{s})=
	\left(f_{\theta _n}^{(0) }(\mathbf{s}),\dots ,f _{\theta _n}^{(N-1) }(\mathbf{s})
	\right)\right\}_{n \geq 1}.$$
The $i$-th component $f _{\theta _n}^{(i) }(\mathbf{s})$
of $\mathbf{f}_{\theta _n}(\mathbf{s})$ is
$$
	f _{\theta _n}^{(i) }(\mathbf{s})=\sum_{\mathbf{ j}
		\in \mathbb{N}_0^N}
	f _{\theta _n}^{(i) }[\mathbf{j}]\cdot
	\prod _{k=0}^{ N-1}s _{k}^{j_k }\text{ for }
	\mathbf{s}=(s_0,\dots ,s_{N-1})\in [0,1]^{N},\,
	\mathbf{j}=(j_0,\dots ,j_{N-1})
$$
where $f _{\theta _n}^{(i) }[\mathbf{j}]$ is the probability that a level $n-1$ individual of type $i$ gives birth to $j_k$ individuals of type $k$ in environment $\theta _n$.

If we condition on a realization $\overline{\pmb{\theta}}=\left\{\theta _n\right\}_{n \geq 1}$ of the environmental process
then $$\left\{\mathbf{Z}_n(\overline{\pmb{\theta}} )=
	\left(Z _{n}^{(0) }(\overline{\pmb{\theta}}),\dots ,Z _{n}^{(N-1) }(\overline{\pmb{\theta}})\right)\right\}_{n=1}^{\infty  }$$
behaves like an $N$-dimensional temporally non-homogeneous branching process, where $Z _{n}^{(i) }(\overline{\pmb{\theta}})$ is the number of type
$i$ individuals in the $n$-th generation in environment $\overline{\pmb{\theta}}$. In this way the offspring distribution of a type $i$ individual in the $n$-th generation is given by $f_{\theta_{n+1}}^{(i)}$.
For a $\theta _n$ we consider the ($N\times N$) expectation matrix $\mathbf{M}_ {\theta _n}$,
$$
	\mathbf{M}_{\theta _n}(i,j):=\frac{\partial f_{\theta _n}^{(i)}}{\partial s_j}(\mathbf{1}),
$$
where $\mathbf{1}\in\mathbb{R}^N$ is the vector with all components equal to $1$. For a non-negative matrix $\mathbf{M}$ let $\|\mathbf{M}\|$ denote the sum of all of its elements.
The Lyapunov exponent corresponding to the matrices and the environmental sequence is 
\begin{equation*}
	\lambda :=\lim\limits_{n\to \infty }
	\frac{1}{n}\log \| \mathbf{M}_{\theta _1}\cdots \mathbf{M}_{\theta _n} \|,
\end{equation*}
where the limit exists and is the same constant for almost all
$\left\{\theta _n\right\}_n$.

\subsection {Our result in a special case}
In the special case when we assume that $\overline{\pmb{\theta}}=\left\{\theta _n\right\}_{n \geq 1}$ is a stationary ergodic process over a finite alphabet, our result implies the following:
\begin{theorem}\label{w77}
	Assume that
	\begin{enumerate}
		[label=\emph{(\alph*)}]
		\item $\overline{\pmb{\theta}}=\left\{\theta _n\right\}_{n \geq 1}$ is a stationary ergodic process over a finite alphabet $[L]:=\left\{0,\dots ,L-1\right\}$.
		\item For every $\theta \in [L]$
		      the expectation matrix $\mathbf{M}_{\theta}$ is allowable, i.e. it contains a strictly positive number in each row and column.
		\item There exists an $n$ and a $(\theta _1,\dots ,\theta _n)\in [L]^n$ such that
		      all elements of the product matrix $\mathbf{M}_{\theta _1}\cdots\mathbf{M}_{\theta _n}$ are strictly positive.
		\item There exists a $K<\infty$ such that for every $\theta \in [L]$ and $i,j\in [N]$,
		      $
			      \frac{\partial^2 f_{\theta}^{(k)}}{\partial s_j \partial s_i}(\mathbf{1})<K
		      $.
	\end{enumerate}
	Under these conditions we have:
	\begin{enumerate}
		\item If $\lambda > 0$ then in almost every environment $\overline{\pmb{\theta}}$, starting with a single individual of arbitrary type,
		      the process $\mathbf{Z}_n(\overline{\pmb{\theta}} )$ does not die out with positive probability. Moreover, $\lim_{n\to\infty} n^{-1}\log(\|\mathbf{Z}_n\|)=\lambda$, conditioned on the process does not die out. \label{x49}
		\item If $\lambda < 0$ then
		      in almost every environment $\overline{\pmb{\theta}}$, starting with a single individual of arbitrary type, the process $\mathbf{Z}_n(\overline{\pmb{\theta}} )$ does die out almost surely.\label{x48}
		\item If $\lambda=0$ and with positive probability we can find $\theta$ such that for every type $i$ the probability that a type $i$ individual gives birth to only $0$ or $1$ child is less than $1$ (i.e. for all $i$ $f _{\theta _n}^{(i) }[\mathbf{0}]+\sum_{j\in[N]}f _{\theta _n}^{(i) }[\mathbf{e}_j]<1$) then starting with a single individual of arbitrary type, the process $\mathbf{Z}_n(\overline{\pmb{\theta}} )$ does die out almost surely.\label{x47}
	\end{enumerate}
\end{theorem}
In the text, for part \ref{x49} see Corollary \ref{x46} and \ref{x54}, for parts \ref{x48} and \ref{x47} see Corollary \ref{x45} and Fact \ref{x44}.

\subsection{Corresponding literature}
The extinction problem for MBPREs was investigated by Athreya and Karlin in \cite[Theorem 8]{bp_renv}. They proved that under some conditions in almost every environment:
\begin{equation}
	\label{w79}
	\lambda < 0 \implies \text {almost sure extinction, and }
	\lambda > 0 \implies \text {survival with positive probability.}
\end{equation}
The conditions of \cite[Theorem 8]{bp_renv} include the assumption that for every $\theta $ and $i,j$ we have
$\mathbf{M}_\theta (i,j)>0$. At the same year Weissner (\cite{zbMATH03349140})
and later Tanny (\cite{zbMATH03738673}) weakened this assumption---they required that there exists a $k$
such that for all $\theta _1,\dots ,\theta _k$ of positive probability
the corresponding product of the expectation matrices is strictly positive, i.e.
\begin{equation}\label{x58}
\exists k,\; \forall \theta_1, \dots, \theta_k\text{ with } \, \nu(\theta_1, \dots, \theta_k)>0,\,\forall i, j :\; (\mathbf{M}_{\theta_1 }\cdots \mathbf{M}_{\theta_k})(i,j)>0,
\end{equation}
where $\nu$ is the distribution of the environmental sequence (see Definition \ref{z63}). This condition always fails in the case that at least one of the matrices is triangular and the environment is an i.i.d.\ sequence (assuming that every letter has positive probability).
In particular this happens in our motivating example described in Section \ref{w78}. In this note we weaken the positivity assumptions. More precisely, we assume that 
\begin{itemize}
	\item each of the expectation matrices are ``uniformly'' allowable (see Definition \ref{x33}, which always holds when we assume that the underlying alphabet is finite, and all of our matrices have a strictly positive element in every row and every column) and 
	\item we require (Cf. \eqref{x58})
\begin{equation}\label{x43}
\exists k,\; \exists \theta_1, \dots, \theta_k\text{ with } \, \nu(\theta_1, \dots, \theta_k)>0,\,\forall i, j :\; (\mathbf{M}_{\theta_1 }\cdots \mathbf{M}_{\theta_k})(i,j)>0.
\end{equation}
\end{itemize}

\subsection{Application to random self-similar sets}\label{w78}

Our motivation to consider multitype branching processes in random environments (MBPREs) comes from the theory of random self-similar sets, of which we show an example below. In this application the positivity conditions described in the previous section (e.g. \eqref{x58}) typically do not hold.

One example of a random self-similar set to which our theorem applies is the $45$-degree projection of the random Sierpi\'nski carpet. 
The positivity of Lebesgue measure of this projection depends on the extinction probability of a corresponding multitype branching processes in finitely generated random environment (where at each step we choose the driving distribution from the same finite family of distributions uniformly and independently). The explanations corresponding to this connection between the Lebesgue measure of the projection of the random Sierpi\'nski carpet and an MBPRE given in this section are heuristic, their purpose is to enlighten the underlying ideas.

The random Sierpi\'nski carpet was introduced in \cite[Example 1.1]{dekking1990structure}, projections of similar random constructions has been studied in for example \cite{falconer1989}, \cite{zbMATH00064453} and in particular positivity of Lebesgue measure of projections (only to the coordinate axes) of such sets has been considered in \cite{zbMATH04022321}. The size (existence of interior points and positivity of Lebesgue measure) of algebraic differences of 1-dimensional random Cantor sets has been investigated in \cite{zbMATH05255663}, \cite{zbMATH05643268}.

\subsubsection*{Deterministic Sierpi\'nski carpet}

The deterministic Sierpi\'nski carpet is the attractor of the IFS 
$$
	\mathcal{S}=\{S_i(\underline{x})=\frac{1}{3}\underline{x}+t_i\}_{i=0}^{7},
$$
in $\R^2$, where $\{t_i\}_{i=0}^{7}$ is the enumeration of the set
$
	\left\{0,1/3,2/3\right\}^2\setminus\left\{ 1/3,1/3\right\}.
$ For the first level approximation, see Figure \ref{x37}.

\begin{figure}
	\centering
	\begin{subfigure}[b]{0.30\textwidth}
		\centering
		\includegraphics[width=\linewidth]{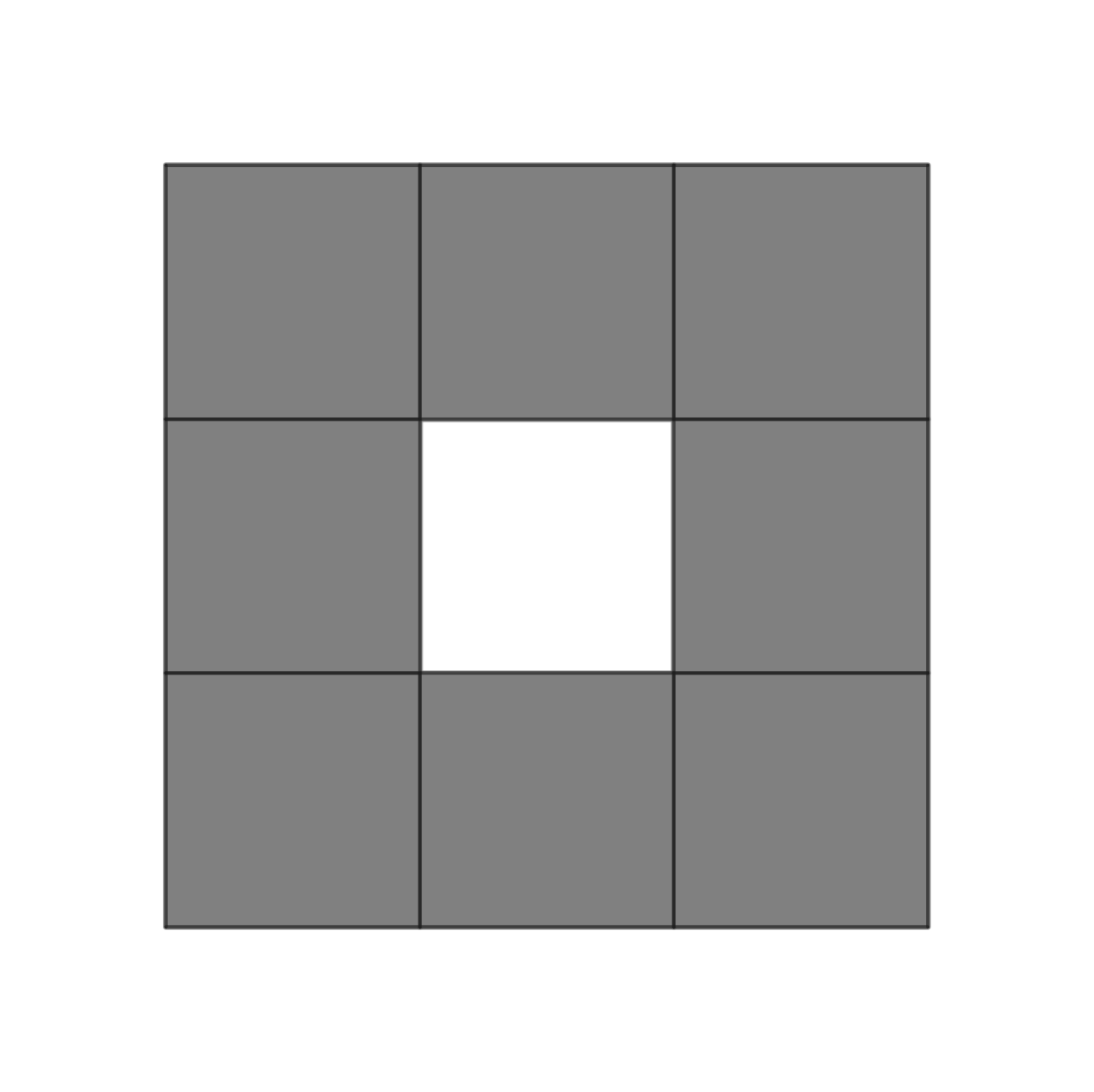}
		\caption{Sierpi\'nski carpet.}\label{x37}
	\end{subfigure}
	\begin{subfigure}[b]{0.30\textwidth}
		\includegraphics[width=\linewidth]{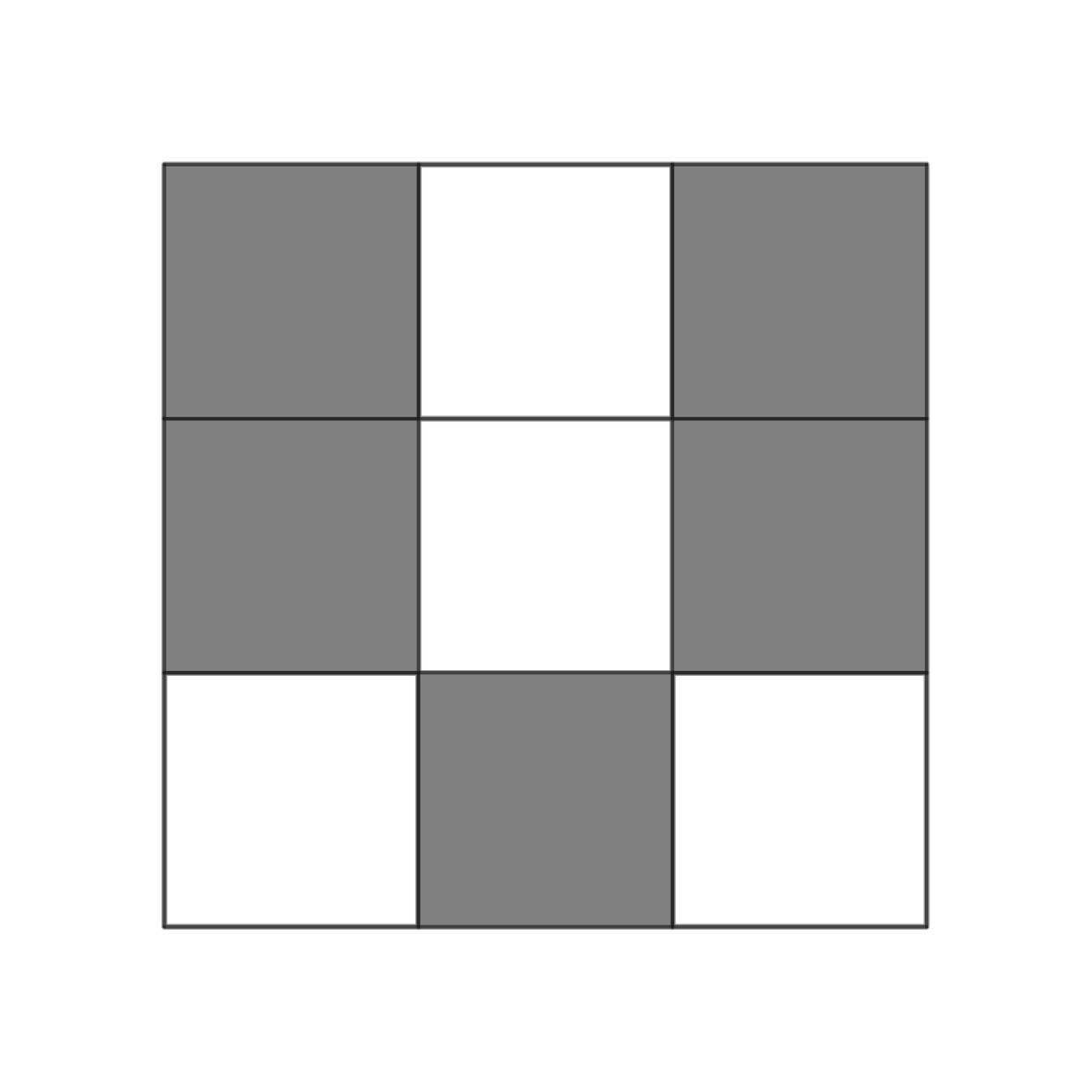}
		\caption{Realization (level 1).}\label{x36}
	\end{subfigure}
	\begin{subfigure}[b]{0.30\textwidth}
		\centering
		\includegraphics[width=\linewidth]{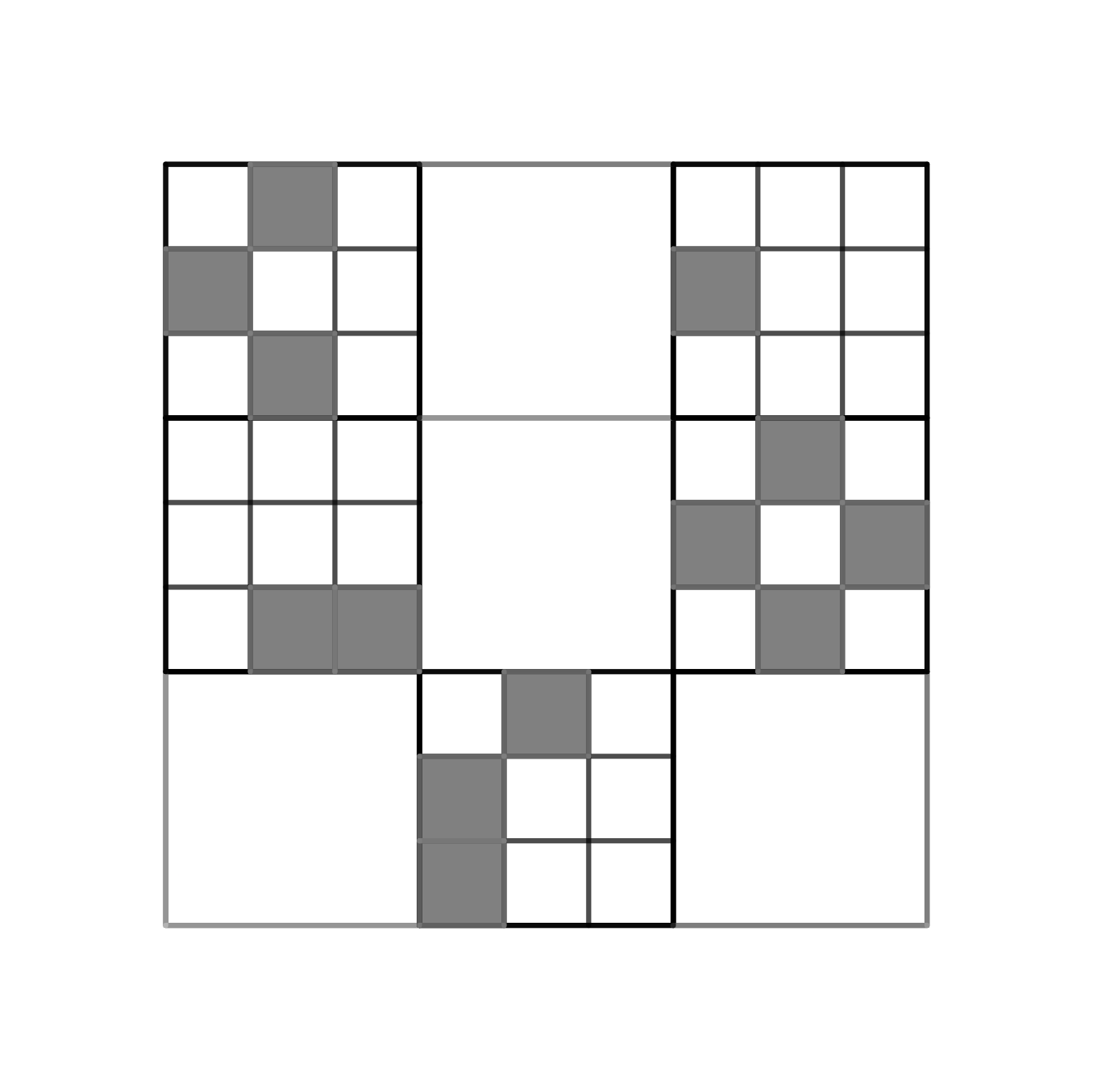}
		\caption{Realization (level 2).}\label{x35}
	\end{subfigure}
	\caption{}
\end{figure}

\subsubsection*{Random Sierpi\'nski carpet}

Now we give an informal description of how to obtain the random Sierpi\'nski carpet. First, we fix a probability parameter $p\in[0,1]$. 
In a given square we repeat the following two steps:
\begin{itemize}
	\item  We subdivide the square into 9 congruent subsquares and discard the middle one to archive the first level approximation of the deterministic Sierpi\'nski carpet. (For the first level approximation of the deterministic Sierpi\'nski carpet see Figure \ref{x37}.)
	\item Each of the remaining 8 congruent squares are retained with probability $p$ and discarded with probability $1-p$ independently of each other. (For a possible realization of this step see Figure \ref{x36}).
\end{itemize}
We start with the $[0, 1]^2$ unit square, and repeat the above described process in the retained cubes independently of each other ad infinitum, or until there are no cubes left. For a realization of a first and second level approximation see Figure \ref{x36} and \ref{x37}.
Formally the construction is defined for example in \cite[Definition 1.1]{OurPaper}.

\begin{figure}
	\centering
	\begin{subfigure}[b]{0.31\textwidth}
		\centering
		\includegraphics[width=\linewidth]{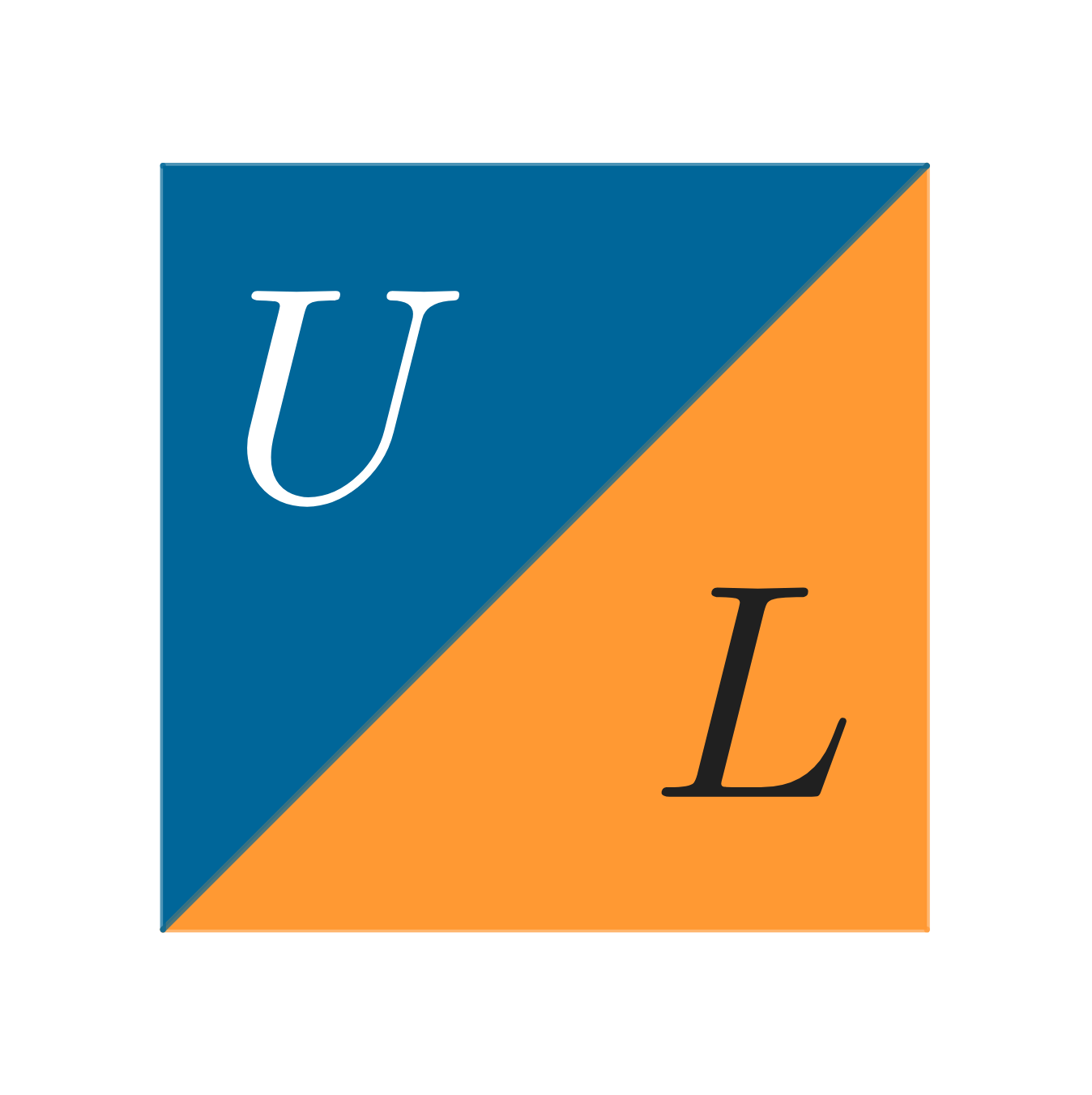}
		\caption{Upper and lower triangle.}\label{x27}
	\end{subfigure}
	\begin{subfigure}[b]{0.31\textwidth}
		\includegraphics[width=\linewidth]{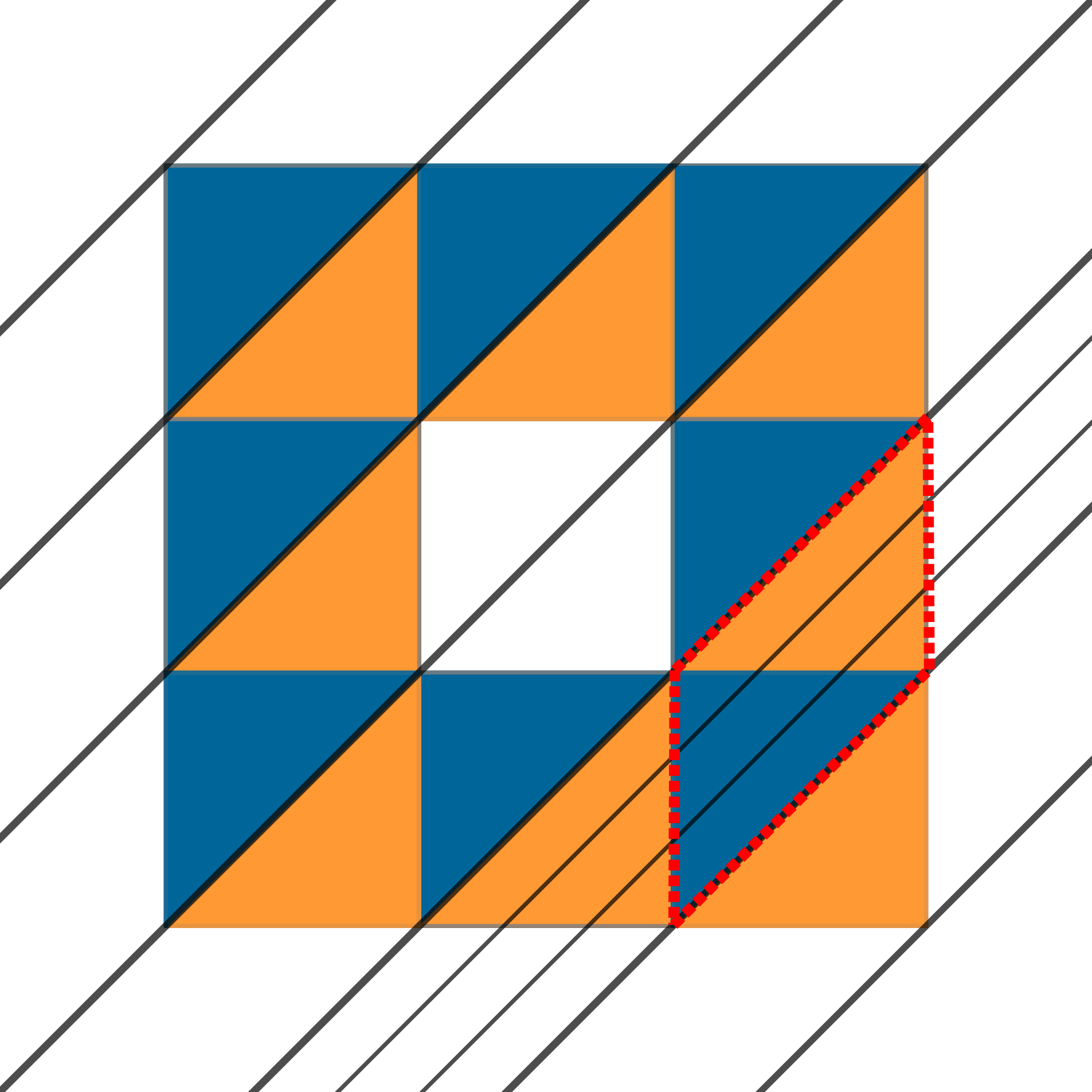}
		\caption{Types and columns.}\label{x26}
	\end{subfigure}
	\begin{subfigure}[b]{0.31\textwidth}
		\centering
		\includegraphics[width=\linewidth]{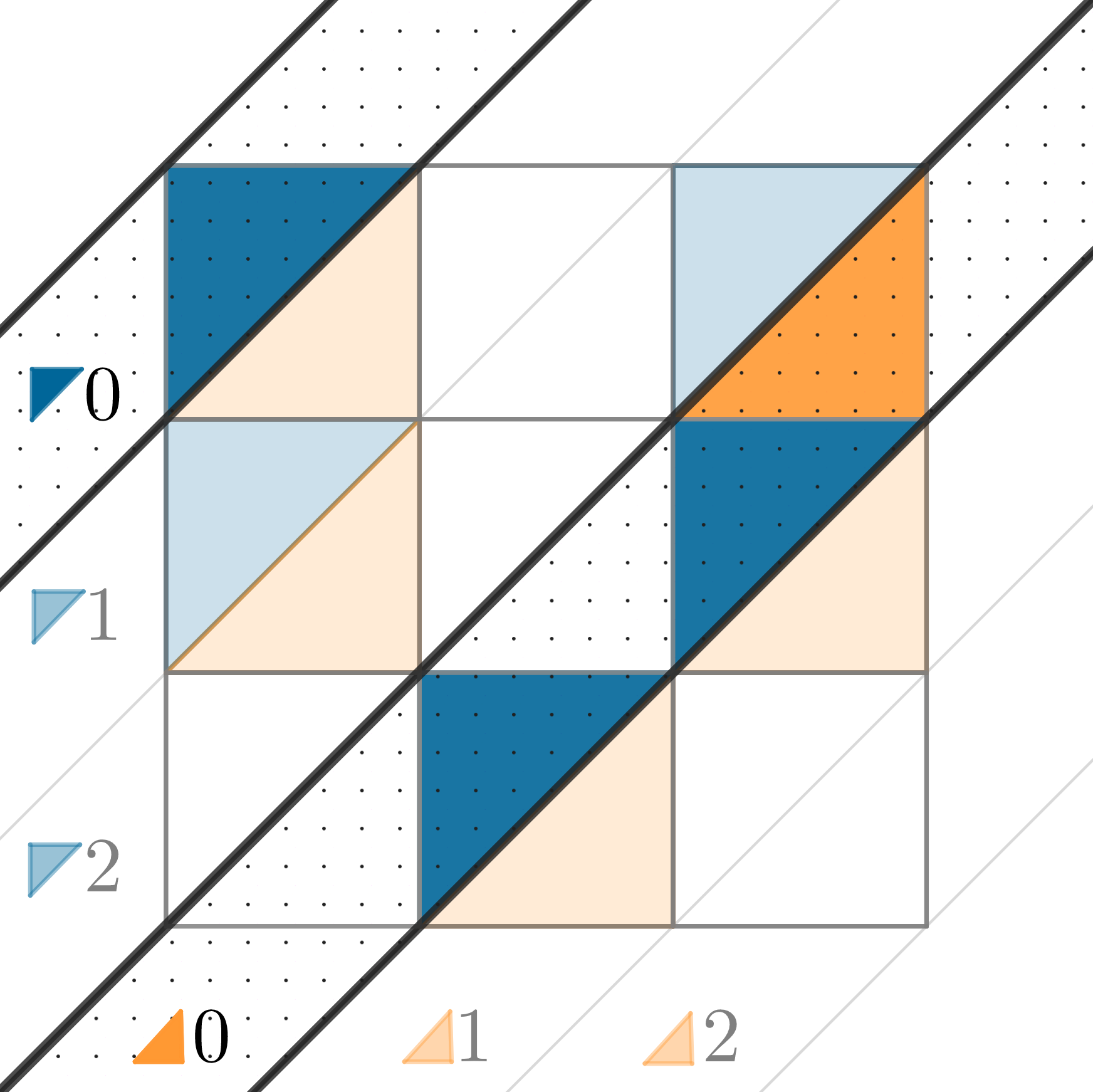}
		\caption{Zeroth column.}\label{x25}
	\end{subfigure}
	\caption{}
\end{figure}

\subsubsection*{Projection and multitype branching processes in random environments}
To investigate the $45$-degree projection of the random Sierpi\'nski carpet it is more practical, instead of analysing the process in squares, to subdivide the square into two triangles as shown in Figure \ref{x27} for the zeroth and \ref{x26} for the first level triangles (in the deterministic Sierpi\'nski case, before the randomization) and analyse the process in terms of triangles. This strategy can be familiar from for example \cite{zbMATH05255663} and \cite{zbMATH05643268}. We now introduce the corresponding MBPRE, starting with types and the environments.
The two level-0 triangles, $U$ (as upper) and $L$ (as lower) will be the prototypes of the two abstract type $\mathfrak{U}$ and $\mathfrak{L}$, which are the types of all the upper and lower triangles (smaller as well) respectively for deeper levels as well. The 6 level-1 columns are the stripes which are bounded by 2 consecutive of the 7 $45$-degree lines going through the vertices (the lines are the thicker lines of Figure \ref{x26}).

The columns are the building blocks of the environment. 
In the first level both zeroth level triangles ($U$ and $L$) subdivides into 3 columns, containing $1/3$-rd smaller upper and lower triangles (Figure \ref{x25}). For each level-1 triangle we consider the corresponding level-2 sub-columns in an analogous way---the lines defining the level-2 columns (for the level-1 upper and lower triangles in the framed area) are drawn with more slender lines. In levels deeper than the zeroth the upper and lower triangles in a given column are always positioned in a way, that their sub-columns pairwise coincide, hence we identify the 3 sub-column of the upper and the 3 sub-column of the lower triangle already in the zeroth level, resulting in 3 rather than 6 columns altogether. For example, in the randomized case in Figure \ref{x25} the zeroth column in a particular realization is highlighted. These columns serve as the alphabet from which we choose the infinite words, the environments.

The particular realization is the first level of a random Sierpi\'nski carpet from Figure \ref{x36} which we further inspect now using Figure \ref{x25}. We focus on the zeroth column. The level-1 upper triangle contains one level-2 upper triangle with probability $p$ (this is the case in this realization) and nothing with probability $1-p$. The zeroth column of the lower triangle however contains two upper and one lower triangle in this realization. Here the number of upper as well as the number of lower triangles follows a binomial distribution with parameter $2$ and $p$. 
It follows that in the zeroth-column the \textit{expected number} of small upper triangles in an upper triangle is $p$, and in a lower triangle is $2p$. The expected number of small lower triangles and of small upper triangle in a lower triangle is both $2p$. We can summarize the last information in an \textit{expectation matrix}, indexed with zero to denote that we are in the zeroth column, 
$$
\mathbf{M}_0=
p\begin{bmatrix}  
1 & 0 \\  
2 & 2  
\end{bmatrix}
$$
in which the first row corresponds to the bigger upper, while the second to the bigger lower triangle, and the first column is to the number of smaller upper and the second is to the number of lower triangles that born from the bigger types. For example the first element of the second row gives the expected number of small upper triangles in a big lower triangle.

The expectation matrices corresponding to the other two columns are
$$
\mathbf{M}_1=p\begin{bmatrix}  
2 & 1 \\  
1 & 2  
\end{bmatrix}, \quad 
\mathbf{M}_2=p\begin{bmatrix}  
2 & 2 \\  
0 & 1 
\end{bmatrix}.
$$

When the survival probability of the MBPRE, informally described above (with two types, the type of upper and the type of lower triangles, and environments described by infinite sequences of columns), is positive then the Lebesgue measure of the $45$-degree projection of the Sierpi\'nski carpet is also positive (almost surely conditioned on the set being non-empty).
The assumptions of Theorem \ref{w77} are satisfied: the environmental process (uniform and i.i.d.\ over the alphabet $[2]$) is stationary and ergodic, the expectation matrices are allowable, $\mathbf{M}_1$ is strictly positive, and assumption (d) is trivial. It is also clear, that for all $n\in\N$, we have $\mathbf{M}_0^n(1,2)=0$. Therefore, condition \eqref{x58} does not hold and the results preceding this paper are not applicable.

The Lyapunov exponent ($\lambda$) in this case is estimated by Pollicott and Vytnova \cite{pollicottVytnova} to be $1.367< \lambda < 1.395$, hence for the critical value, $p_*$ (i.e.\ for $p>p_*$ we have and for $p<p_*$ we do not have positive Lebesgue measure), $p_*\in [0.247833,0.25487]$.

\subsection{Structure of the paper}
For the remainder of this section, we introduce the notation that we will use throughout this paper.
Then, in the second section, we define multitype branching process in varying and in random environments, and we state our main results precisely.
The remainder of the paper is devoted to the proof of the main theorem.

Our introduction of multitype branching processes in a random environment
in Sections \ref{w88}-\ref{w89}
follows closely that of
\cite[Chapter 10]{Kersting2017} with slight modifications in the notation.
\subsection{Notation}\label{w88}

For $k>0$, $[k]:=\{0, \dots, k-1\}$.
We denote the vectors and matrices by boldface letters; in particular
\begin{align*}
	\mathbf{0}:=(0, \dots, 0)\text{, }
	\mathbf{1}:=(1, \dots, 1)\text{ and }
	\mathbf{e}_i:=(\underbrace{0, \dots,0}_{i}, 1,\underbrace{ 0, \dots, 0}_{N-i-1}).
\end{align*}
For two $N$-dimensional vectors $\mathbf{u}=(u_0, \dots, u_{N-1})$ and $\mathbf{v}=(v_0, \dots, v_{N-1})$ and the $N\times N$ matrices $\mathbf{U}=(u_{i,j})_{i,j\in[N]}$ and $\mathbf{V}=(u_{i,j})_{i,j\in[N]}$ let
\begin{align}
	 & \mathbf{u}\cdot \mathbf{v}:= u_0v_0+ \dots+ u_{N-1}v_{N-1} ,\,
	\mathbf{u}^{\mathbf{v}}=\prod_{i=0}^{N-1}u_{i}^{v_i}       \text{ and } \\
	 &
	\begin{array}{lr}
		\mathbf{u} \leq \mathbf{v} \\
		\mathbf{U}\leq \mathbf{V}
	\end{array}\text{ if and only if }
	\begin{array}{lr}
		u_i\leq v_i \\
		u_{i,j}\leq v_{i,j}
	\end{array}\text{ for all $i,j\in [N]$}.\label{w76}
\end{align}
We further use the strict equality version of \eqref{w76}, when all $\leq$ is replaced with $<$.
Given the functions $\mathbf{f}=(f^{(0)}, \dots, f^{(N-1)}), \{\mathbf{f}_i=(f^{(0)}, \dots, f^{(N-1)})\}_{i\in \mathcal{I}}$ and $g$, $\mathbf{f},\mathbf{f}_i:\R^N\to\R^N$ and $g:\R^N\to\R$ for some $N\in \N$ we write
\begin{itemize}
	\item for $K\in\R, K\geq 0$ $g^K=g\cdot \dotsc \cdot g$ and
    \item for $0 \leq \mathbf{K}=(K_0,\dots, K_{N-1})\in\R^N$, $\mathbf{f}^K=(f^{(0)})^{K_0}\cdot \dotsc \cdot (f^{(N-1)})^{K_{N-1}}$;
    \item for $\underline{i}=(i_1, \dots, i_n)\in\mathcal{I}^{n}$,
$f_{\underline{i}}:=f_{i_1}\circ \cdots \circ f_{i_n}$.
\end{itemize}

Some further notation we use throughout paper, including the place of the first occurrence:

\begin{center}
	\setlength\extrarowheight{4pt}
	\begin{longtable}{@{}m{0.14\textwidth} | m{0.72\textwidth} m{0.05\textwidth}}
	symbol& explanation&link\\
	\hline

	$\mathscr{A}$&$N\times N$ non-negative, allowable matrices&  Sec. \ref{z49}\\

	$\lambda$&Lyapunov exponent corresponding to the expectation matrices and the ergodic measure $\nu$&  Def. \ref{x84}\\
																										
	$\lambda_*$&column-sum exponent corresponding to the expectation matrices and the ergodic measure $\nu$& Def. \ref{x83}\\
																						
	$\alpha$ &the uniform allowability constant&\eqref{z15}\\

	$\mathbf{M}_{\theta}$&expectation matrix in the environment $\theta$& \eqref{eq:z71}	\\
																								
	$\mathbf{A}_{\theta}, \rho$ & $\mathbf{A}_{\theta}$ is the $\rho$-decreased expectation matrix &\eqref{x21}	\\
																										
	$\mathfrak{W},\,\mathfrak{W}^{\theta, k}$&$\left\{(k,i,\theta)\in [N]^2\times\mathcal{I}: \mathbf{M}_{\theta}(k,i)>0 \right\}$ and $\left\{i\in [N]: (k,i,\theta)\in \mathfrak{W} \right\}$& \\

	$t_{\theta}^{(k)}(\mathbf{s})$& $1-\mathbf{r}_{k}(\mathbf{M}_{\theta})\cdot (\mathbf{1}-\mathbf{s})$&\eqref{x62}	\\

	$\mathbf{r}_k(\mathbf{B}), \mathbf{c}_k(\mathbf{B})$&the $k$-th row and column vector of a matrix $\mathbf{B}$, respectively & \\

	$g_{\theta}^{(k)}(\mathbf{s}), \, \mathbf{g}_{\theta}(\mathbf{s})$&$1-\mathbf{r}_{k}(\mathbf{A}_{\theta})\cdot (\mathbf{1}-\mathbf{s})$  and $(g_{\theta}^{(0)}(\mathbf{s}), \dots, g_{\theta}^{(N-1)}(\mathbf{s}))\!=\!\mathbf{1}-\mathbf{A}_{\theta}(\mathbf{1}-\mathbf{s})$ resp.&\eqref{x90}\\
																									  
	$B_{\delta}$&$\{\mathbf{s}\in [0,1]^N: \|\mathbf{1}-\mathbf{s}\|_{\infty}\leq \delta\}$ & \eqref{x99}\\
																									  
	$\psi(\mathbf{s})$&& \eqref{z34}\\
																								  
	$R(t),\, R^C(t)$&$\{ \mathbf{s} \in [0,1]^{N}:\|\mathbf{s}\|<t\} \text{ and } \{ \mathbf{s} \in [0,1]^{N}:\|\mathbf{s}\|\geq t\}$ resp.&\eqref{x87}\\

	$\varphi_v (t)$&$N-N v+v t$&\eqref{x88}\\
																									  
	$q^{(k)}(\overline{\pmb{\theta}}), \, \mathbf{q}(\overline{\pmb{\theta}})$&the probability that the process starting with one individual of type-$k$ becomes extinct, and the vector of these, $(q^{(k)}(\overline{\pmb{\theta}}))_{k\in[N]}$ resp.& \eqref{x59}\\
																          
	\end{longtable}                                                                               
	\end{center}

	\section{Introduction to multitype branching processes}
We now introduce some preliminary notation regarding the $N$-type branching process for $N\geq 2$.
Denote by $\mathcal{P}(\mathbb{N}_{0}^N)$ the probability distributions on $\mathbb{N}_{0}^N$. Furthermore, we identify the distributions in $\mathcal{P}(\mathbb{N}_{0}^N)$ with their probability generating functions (pgfs) $f$. That is let $\mu \in \mathcal{P}\left(\mathbb{N}_{0}^N\right)$ then we identify $\mu $ with
its pgf,
\begin{equation}
	f(\mathbf{s}):=\sum_{\mathbf{z}\in \mathbb{N}_{0}^N}f[\mathbf{z}]\mathbf{s}^\mathbf{z},
	\text{ for }
	\mathbf{s}\in[0,1]^N,
\end{equation}
where $f[\mathbf{z}]:=\mu (\left\{\mathbf{z}\right\} )$, $\mathbf{z}=(z_0, \dots, z_{N-1})\in \mathbb{N}_0^N$.

We consider
the $N$-dimensional vectors of
probability measures which are identified by
the vectors of their pgf
\begin{equation}\label{z13}
	\mathbf{f}=(f^{(0)}, \dots, f^{(N-1)})\in \mathcal{P}(\mathbb{N}_{0}^N)\times \cdots \times \mathcal{P}(\mathbb{N}_{0}^N)=\mathcal{P}^N(\mathbb{N}_{0}^N).
\end{equation}

\subsection{Multitype branching processes in varying environments}\label{f99}
On the underlying probability space $(\Omega ,\mathcal{F},\mathbf{P})$
we define the $N$-type branching process in varying environment for an $N \geq 2$.
 
\begin{definition}\label{z25}
	A sequence $\overline{v}=(\mathbf{f}_1, \mathbf{f}_2, \dots)$,  of $N$-dimensional probability measures $\mathbf{f}_i = (f^{(0)}_i, \dots, f^{(N-1)}_i)$ (see \eqref{z13}) on $\mathbb{N}_{0}^N$ is called a \texttt{varying environment}.
\end{definition}

\subsubsection{Alternative description of the process}\label{z95}
	Assume that we are given a varying environment $\overline{v}=\left(\mathbf{f}_n\right)_{n\geq 1}$ of $N$-dimensional probability measures.
	For each $i\in[N]$ and $n \geq 1$ there is an offspring vector random variable
	\begin{equation*}
	\mathbf{Y} _{n}^{(i)}=(Y _{n}^{(i) }(0), \dots ,Y _{n}^{(i)}(N-1))
 	\end{equation*}
	such that
 	\begin{equation*}
 	\mathbf{P}\left(\mathbf{Y}_{n}^{(i)}=\mathbf{y}\right)=f _{n}^{(i) }[\mathbf{y}] ,\, \text{ for every } \mathbf{y}\in\mathbb{N} _{0}^{N}.
	\end{equation*}
Now we define $\left\{ \mathbf{Z}_n\right\}_{n\geq 0}$ the \texttt{N-type branching process in the varying environment} $\overline{v}$.
We start at level $0$, where the number of different types of individuals is deterministic and is given by
$\mathbf{z}_0:=(z _{0}^{(0) },\dots ,z _{0}^{ (N-1)})$, that is
$\mathbf{Z}_0:=\mathbf{z}_0$. Given $\mathbf{Z}_0, \dots, \mathbf{Z}_{n-1}$ we define $\mathbf{Z}_n$ as follows.

We consider the sequence of vector random variables
		      $$
			  \left\{\mathbf{Y}_{j,n}^{(i)}
			      =(Y _{j,n}^{(i) }(0),\dots Y _{j,n}^{(i) }(N-1)) : i\in [N],\ j\in\{1,\dots, \mathbf{Z}_{n-1}^{(i)}\}\right\},
		      $$
		      
			  \begin{enumerate}[label={(\alph*)}]
				\item $\left\{\mathbf{Y}_{j,n}^{(i)}\right\}_{i,j}$ are independent of each other and $\mathbf{Z}_{n-1}$, and 
				\item $\mathbf{Y}_{j,n}^{(i)}\stackrel{d}{=}\mathbf{Y}_{n}^{ (i)}$.
			  \end{enumerate}

			  Informally the meaning of the $\ell $-th component,
			  $Y_{j,n}^{(i)}(\ell )$ of $\mathbf{Y}_{j,n}^{(i)}$ is the number of type-$\ell$ individuals of level $n$ given birth by the $j$-th level $n-1$ type-$i$ individuals.
			  
Then the vector of the numbers of various type level-$n$ individuals is
\begin{equation} \label{eq:z72}
	\mathbf{Z}_n=(Z_n^{(0)}, \dots, Z_{n}^{(N-1)}):= \sum_{i=0}^{N-1}\sum_{j=1}^{Z_{n-1}^{(i)}}\mathbf{Y}_{j,n}^{(i)},
\end{equation}
where $Z _{n}^{(i)}$ stands for the number of type $i$ individual in the $n$-th generation.
\begin{definition}\label{x93}
Formally, the stochastic process $\mathcal{Z}=\{\mathbf{Z}_n\}_{n\geq 0}$ is called
\texttt{$N$-type branching process with varying environment $\overline{v}$} if for any $\mathbf{z}\in\N_{0}^N$
\begin{equation*}
	\mathbf{P}_{\mathbf{z}_0,\overline{v}}
	(\mathbf{Z}_n=\mathbf{z}|\mathbf{Z}_1, \dots, \mathbf{Z}_{n-1})=(\mathbf{f}_n^{\mathbf{Z}_{n-1}})[\mathbf{z}],
\end{equation*}
where we write $\mathbf{P}_{\mathbf{z}_0,\overline{v}}$ instead of $\mathbf{P}$
to emphasize the initial population size $\mathbf{z}_0$ and the fixed (deterministic) environment $\overline{v}$.
\end{definition}
The two description is connected by Fact \ref{x79} in Appendix \ref{x73}.

\subsection{Multitype branching processes in a random environment}\label{w89}
In this section we describe the generalization of the above process, by instead of considering a fixed deterministic environment we consider random environments. Conditioning on the environment, the process behaves as a multitype branching process in a varying environment.
We endow $\mathcal{P}^N(\mathbb{N}_{0}^N)$ with the metric of total variation (see \cite[p. 260]{Kersting2017}) and with the respective Borel $\sigma $-algebra. Hence, we can speak about
random $N$-dimensional probability measures. These are random variables taking values in $\mathcal{P}^N(\mathbb{N}_{0}^N)$ of the form
$$\mathbf{F}=(F^{(0)}, \dots, F^{(N-1)}),$$
where the components are the pgfs
\begin{equation}\label{z12}
	F^{(i)}(\mathbf{s}):=\sum_{\mathbf{z}\in \mathbb{N}_{0}^N}F^{(i)}[\mathbf{z}]\mathbf{s}^\mathbf{z}, \, i\in[N].
\end{equation}
\begin{definition}\label{z73}
	A \texttt{random environment} is
	a sequence $\mathcal{V}=(\mathbf{F}_n)_{n\geq 1}$ of $N$-dimensional random probability measures taking values in $\mathcal{P}^N(\mathbb{N}_{0}^N)$.
\end{definition}
Now we introduce multitype branching processes $\mathcal{Z}$ in random environments as follows: first we consider a random environment $\mathcal{V}=(\mathbf{F}_n)_{n\geq 1}$. For a realization $\overline{v}=\left(\mathbf{f}_n = (f^{(0)}_n, \dots, f^{(N-1)}_n)\right)_{n\geq 1}$ of $\mathcal{V}$, $\mathcal{Z}$ evolves as an $N$-dimensional temporally non-homogeneous branching process, where
the offspring distribution of a type
$i$ individual on the $n-1$-th generation is governed by $f^{(i)}_n(\mathbf{s})$.

\begin{definition}[MBPRE]\label{z70}
	We say that the process $\mathcal{Z}=(\mathbf{Z}_{n}=(Z^{(0)}_n, \dots, Z^{(N-1)}_n))_{n\in \mathbb{N}}$
	taking values in $\mathbb{N}_{0}^N$
	is a \texttt{multitype ($N$-type) branching process in the random environment $\mathcal{V}=(\mathbf{F}_n)_{n\geq 1}$} (MBPRE)
	if for each realization $\overline{v}=(\mathbf{f}_n)_{n\geq1}$ of $\mathcal{V}$ and for each $\mathbf{z}_0, \mathbf{z_1}, \dots, \mathbf{z_k}\in \mathbb{N}_{0}^N$
	\begin{multline*}
		\mathbb{P}(\mathbf{Z}_1=\mathbf{z}_1, \dots, \mathbf{Z}_k=\mathbf{z}_k|\mathbf{Z}_0=\mathbf{z}_0, \mathcal{V}=\overline{v})
		= \mathbf{P}_{\mathbf{z}_0, \overline{v}}(\mathbf{Z}_1=\mathbf{z}_1, \dots, \mathbf{Z}_k=\mathbf{z}_k)\text{  a.s.},
	\end{multline*}
	where $\mathbf{P}_{\mathbf{z}, v}$ denotes the probability measure corresponding to the $N$-type branching process in varying environment $\overline{v}$ with initial distribution $\mathbf{Z}_0=\mathbf{z}$. We write $\mathbb{P}\left(\cdot\right)$ and $\mathbb{E}\left(\cdot\right)$ for the probabilities and expectations in random environments.
 \end{definition}
From this it follows, that for each realization $\overline{v}=(\mathbf{f}_n)_{n\geq1}$ and $\mathbf{z}, \mathbf{z}_0\in \mathbb{N}_{0}^N$
\begin{equation}\label{x94}
	\mathbb{P}(\mathbf{Z}_n=\mathbf{z}|\mathbf{Z}_0=\mathbf{z}_0, \mathbf{Z}_1=\mathbf{z}_1, \dots,\mathbf{Z}_{n-1}=\mathbf{z}_{n-1}, \mathcal{V}=\overline{v})=(\mathbf{f}_n^{\mathbf{Z}_{n-1}})[\mathbf{z}] \text{ a.s.}
\end{equation}
from which we can conclude, that
\begin{equation}\label{x95}
	\mathbb{E}(\mathbf{s}^{\mathbf{Z}_n}|\mathbf{Z}_0=\mathbf{z}_0, \mathcal{V}=\overline{v})=\mathbf{f}_1(\mathbf{f}_2(\dots(\mathbf{f}_n(\mathbf{s})))^{\mathbf{Z}_0}.
\end{equation}

\subsection{Our Principal Assumptions}
From now on we restrict ourselves to the case when the environment is coming from the infinite product of a \textit{countable} set of distributions from $\mathcal{P}^N(\mathbb{N}_{0}^N)$.

Namely, fix a countable set $\left\{ \mathbf{f}_i \right\}_{i \in \mathcal{I}}$, $\mathbf{f}_i \in \mathcal{P}^N(\mathbb{N}_{0}^N)$ indexed by the set $\mathcal{I}$.
This is the set of possible values of $\mathbf{F}_n$. In this way, the random environment $\mathcal{V}$ is a random variable which takes values from $\left\{ \mathbf{f}_i \right\}_{i \in \mathcal{I}}^{\mathbb{N}}$.

It is more convenient to identify the environments with their ``code'' from $\mathcal{I}^{\N}$, and refer to the code instead.
Namely, we define the map $$\Phi :\left\{ \mathbf{f}_i \right\}_{i \in \mathcal{I}}\to \mathcal{I}^{\N}=\Sigma, \, \Phi ( \mathbf{f}_{\theta _1},\mathbf{f}_{\theta _2},\dots ):=(\theta _1,\theta _2,\dots ).$$ 
\begin{definition}\label{z63}
The probability space $(\Sigma, \mathcal{A}, \nu)$ with the shift map $\sigma$ is defined as
\begin{enumerate}[label=\emph{(\alph*)}]
    \item $\Sigma:=\mathcal{I}^{\N}$,
    \item $\mathcal{A}$ is the usual $\sigma$-algebra on $\Sigma$,
    \item $\nu:=\Phi_*\mathfrak{m}$, where $\mathfrak{m}$ is the distribution of the environmental variable, $\mathcal{V}$. That is $\nu (H)=\mathfrak{m}(\Phi ^{-1}H)$, for any Borel set $H\subset \Sigma $;
    \item and for $\overline{\pmb{\theta}}=(\theta_1, \theta_2,\dots)\in\Sigma$, $\sigma(\overline{\pmb{\theta}}):=(\theta_2, \theta_3,\dots)$.
\end{enumerate}
\end{definition}
We will refer to an \textit{environment} as $\overline{\pmb{\theta}}=(\theta _1,\theta _2,\dots )\in\Sigma$ instead of $( \mathbf{f}_{\theta _1},\mathbf{f}_{\theta _2},\dots )$, and we write $\mathbf{Z}_n(\overline{\pmb{\theta}})$ for $\mathbf{Z}_n$ in the environment $\overline{\pmb{\theta}}$.

In our most important application we usually consider the following special case.
\begin{example}\label{z65}
	When $\Sigma=[L]^{\mathbb{N}}$ for some $2\leq L\in \mathbb{N}$ and the environmental sequence $\mathcal{V}$ is i.i.d. then we have a probability vector $\mathbf{p}=(p_1,\dots ,p_{L-1})$ such that $\mathbb{P}(\mathbf{F}_n=\mathbf{f}_k)=p_k$ for all $n$ and $k\in[L]$. In this case the infinite product measure $\nu=(p_0,\dots ,p_{L-1})^{\mathbb{N}} $ on $\Sigma$ corresponds to the distribution of $\mathcal{V}$ via the identification $\Phi$.
\end{example}
From now on we always assume that:
\noindent
\begin{PA}\label{z99}
	The system $(\Sigma, \mathcal{A},\sigma, \nu)$ defined in Definition \ref{z63} is ergodic.
	\label{PA2}
\end{PA}
\subsection{Expectation matrices and survival probabilities}
\subsubsection{Expectation matrices}
We define the $N \times N$ \textit{expectation matrix} corresponding to a fixed $\theta\in \mathcal{I}$ as 
\begin{equation}\label{eq:z71}
	\mathbf{M}_{\theta}(i,k)=\frac{\partial f_{\theta}^{(i)}}{\partial s_k}(\mathbf{1}).
\end{equation}

In case the environment $\overline{\pmb{\theta}}=(\theta_1, \dots, \theta_n, \dots)$ is fixed  using the notations of Section \ref{z95},
\begin{equation*}
	\mathbf{M}_{\theta_n}(i,k)=\mathbb{E}(Y^{(i)}_{n}(k)).
\end{equation*}
From this, using induction it follows that for any $\overline{\pmb{\theta}}=(\theta_1,\theta_2, \dots)\in \Sigma$ and $n\in\N$
\begin{equation*}
	\mathbb{E}\left[
		\mathbf{Z}_n|\mathcal{V}=\overline{\pmb{\theta}}, \mathbf{Z}_0=\mathbf{z}_0
		\right]
	=\mathbf{z}_{0}^{T}\mathbf{M}_{\theta _1}\cdots\mathbf{M}_{\theta _n}.
\end{equation*}

\subsubsection{Survival probabilities}
Fix $\overline{\pmb{\theta}}=(\theta_1, \theta_2, \dots)\in \Sigma$.
For every $\ell \in \mathbb{N}$
we consider the pgf vector
\begin{equation*}
	\mathbf{f}_{\theta _\ell }(\mathbf{s})=
	(f_{\theta _\ell }^{(0)}(\mathbf{s}),\dots ,
	f_{\theta _\ell }^{(N-1)}(\mathbf{s})),\quad
	f_{\theta _\ell }^{(i)}(\mathbf{s})
	=
	\sum_{\mathbf{ j}\in \mathbb{N}_0^N}
	f_{\theta _\ell }^{(i)}
		[\mathbf{j}]\cdot
	\mathbf{s}^{\mathbf{j}}.
\end{equation*}
Recall from the introduction that $f_{\theta _\ell }^{(i)}[\mathbf{j}]$ ($\mathbf{j}=(j_0, \dots, j_{N-1})$) is the probability that a level $\ell -1$ individual of type $i$ gives
birth to
$j_k$ individuals of type $k$ for every $k\in[N]$.

Applying \eqref{x95} to $\mathbf{z}_0=\mathbf{e}_i$ gives that
\begin{equation*}
\mathbb{E}\left[s^{\mathbf{Z}_n}
		|\mathcal{V}=\overline{\pmb{\theta}},
		\mathbf{Z}_0=\mathbf{e}_i
		\right]=f _{\overline{\pmb{\theta}}|_n}^{(i) }(\mathbf{s}).
\end{equation*}
This implies that
\begin{equation}
	\label{w85}
	\mathbb{P}\left(\mathbf{Z}_n(\overline{\pmb{\theta}})=\mathbf{0}|\mathbf{Z}_0=\mathbf{e}_i\right)=
	f _{\overline{\pmb{\theta}}|_n}^{(i) }(\mathbf{0})
	= f _{\theta _1}^{(i) }(\mathbf{f}_{\theta _2}
	\circ\cdots \circ \mathbf{f}_{\theta_n }(\mathbf{0})).
\end{equation}

Let $q^{(k)}(\overline{\pmb{\theta}})$ denote the probability that the process starting with one individual of type-$k$ becomes extinct, and
\begin{equation}\label{x59}
	\mathbf{q}(\overline{\pmb{\theta}})=(q^{(0)}(\overline{\pmb{\theta}}), \dots, q^{(N-1)}(\overline{\pmb{\theta}})).
\end{equation}
Further we denote the level-$n$ extinction probability by
\begin{equation}
	q_{n}^{(k)}(\overline{\pmb{\theta}})=\mathbb{P}(\mathbf{Z}_{n}(\overline{\pmb{\theta}})=\mathbf{0}|\mathbf{Z}_0=\mathbf{e}_k),\text{ and }\mathbf{q}_{n}(\overline{\pmb{\theta}})=(q_{n}^{(0)}(\overline{\pmb{\theta}}), \dots, q_{n}^{(N-1)}(\overline{\pmb{\theta}})).
\end{equation}

Using \eqref{w85} we obtain
\begin{equation}\label{x81}
	\mathbf{q}_{n}(\overline{\pmb{\theta}})=\mathbf{f}_{\overline{\pmb{\theta}}|_n}(\mathbf{0}).
\end{equation}
Hence,
\begin{equation}\label{x80}
	\mathbf{q}(\overline{\pmb{\theta}})=\lim_{n\to \infty}\mathbf{q}_{n}(\overline{\pmb{\theta}})=\lim_{n\to \infty}\mathbf{f}_{\overline{\pmb{\theta}}|_n}(\mathbf{0}).
\end{equation}
\section{Lyapunov and column-sum exponent}\label{z49}
For this and the following subsections we fix
$N\geq 2$.
Let $\mathscr{A}$ be the set of $N\times N$
allowable matrices (there exist some strictly positive elements in every row and every column) with only non-negative element.
For a $\mathbf{B}\in \mathscr{A}$ we introduce the minimum and maximum column sums:
\begin{equation}
	(\mathbf{B})_{*}=\min_{j\in[N]}\sum_{i\in[N]}\mathbf{B}_{i,j},\quad
	\|\mathbf{B} \|_1:=
	\max_{j\in[N]}\sum_{i\in[N]}\mathbf{B}_{i,j}.
\end{equation}
Finally, we define the norm we will mainly use throughout the paper
\begin{equation}
	\label{z53}
	\|\mathbf{B}\|:=\sum\limits_{i\in[N]}\sum\limits_{j\in[N]}\mathbf{B}_{i, j}.
\end{equation}
For $\{\mathbf{B}_{i}\}_{i\in\mathcal{I}}\subset \mathscr{A}$ and 
$\pmb{\theta}=(\theta_1, \dots, \theta_n)\in\mathcal{I}^n$ we denote
\begin{equation*}
	\mathbf{B}_{\pmb{\theta}}:=\mathbf{B}_{\theta_1}\cdots \mathbf{B}_{\theta_n}.
\end{equation*}

\begin{definition}[Good set of matrices]\label{z19}
	Let $\mathcal{B}=\{\mathbf{B}_i\}_{i\in\mathcal{I}}\subset \mathscr{A}$
	and let $\nu$ be an ergodic invariant measure on $(\Sigma,\mathcal{A} ,\sigma)$, where
	$\Sigma :=\mathcal{I}^{\mathbb{N}}$.
	We say that \textit{$\mathcal{B}$ is good} (with respect to $\nu $) if
	\begin{enumerate}
		\item $m_1:=\int |\log \|\mathbf{B}_{\theta _0} \|_1|d\nu (\theta )+
			      \int |\log (\mathbf{B}_{\theta _0})_{*}|d\nu (\theta )
			      <\infty$; \label{z50}
		\item There exists a $\pmb{\theta}=(\theta_1, \dots, \theta_n)\in \mathcal{I}^n$ such that $\nu(\{\overline{\pmb{\theta}}\in \Sigma: \overline{\pmb{\theta}}|_n=\pmb{\theta}\})>0$ and every element of $\mathbf{B}_{\pmb{\theta}}$ is strictly positive.
	\end{enumerate}
\end{definition}
Note that it follows from the fact that $(\mathbf{B})_{*}>0$ if $\mathbf{B}\in\mathscr{A}$ that assumption \eqref{z50} of Definition \ref{z19} always holds when $|\mathcal{I}|<\infty$.

Now we define the Lyapunov exponent of a random matrix product.
\begin{definition}[Lyapunov exponent]\label{x84} We are given an
	ergodic measure $\nu $ on $(\Sigma ,\sigma )$ and a
	$$\mathcal{B}=\{\mathbf{B}_{i}\}_{i\in\mathcal{I}}\subset \mathscr{A},$$ which is good
	with respect to $\nu $.
	The \textit{Lyapunov}-\textit{exponent} corresponding to $\mathcal{B}$ and the ergodic measure $\nu$ is
	\begin{equation*}
		\lambda:=\lambda(\nu, \mathcal{B})= \lim_{n \to \infty} \frac{1}{n} \log\|\mathbf{B}_{\overline{\pmb{\theta}}|_n}\| \; \text{for $\nu$-almost every $\overline{\pmb{\theta}}\in \Sigma$}.
	\end{equation*}
\end{definition}
The existence of $\lambda$ as defined above follows from
\cite[Corollary 10.1.1]{walters2000introduction}, and the ergodicity of $\nu$.

Using the super-multiplicativity of $(\cdot)_*$ for non-negative allowable matrices (namely, if $\mathbf{B}_1, \dots, \mathbf{B}_n\in \mathscr{A}$ then $(\mathbf{B}_1 \cdots \mathbf{B}_n)_* \geq (\mathbf{B}_1)_*\cdots \allowbreak (\mathbf{B}_n)_*$) it follows that we can define the analogue of the Lyapunov exponent for the minimal column sum, which we call the \textit{column-sum exponent}.
\begin{definition}\label{x83}
	The \textit{column-sum exponent} corresponding to a good set of matrices, $\mathcal{B}=\{\mathbf{B}_{i}\}_{i\in \mathcal{I}}$ and an ergodic measure $\nu$ is
	\begin{equation}
		\lambda_{*}:=\lambda_{*}(\nu, \mathcal{B})= \lim_{n \to \infty} \frac{1}{n} \log(\mathbf{B}_{\overline{\pmb{\theta}}|_n})_{*}\text{ for $\nu$-almost every $\overline{\pmb{\theta}}\in \Sigma$}.
	\end{equation}
\end{definition}
Below we give conditions under which $\lambda =\lambda _*$.
\begin{lemma}\label{z60}
	Let $\nu$ be an ergodic measure on $(\Sigma, \sigma)$.
	If $\mathcal{B}=\{\mathbf{B}_i\}_{i\in\mathcal{I}}\subset \mathscr{A}$ is good then
	\begin{equation}
		\lambda(\nu, \mathcal{B})=\lambda_*(\nu, \mathcal{B}).
	\end{equation}
\end{lemma}
\begin{proof}
	The assertion follows from \cite[Theorem 2]{Hennion97}. For more details see the see Appendix \ref{x82}.
\end{proof}

\section{The main theorem. Extinction probability for MBPRE}\label{z29}

Before we state our theorem we define a condition.
\begin{definition}\label{x33}
	We say that the MBPRE (see Definition \ref{z70}) with state space $\mathcal{I}$ is \textit{uniformly allowable} if there exist an $\alpha>0$ such that
	\begin{equation}\label{z15}
		\inf\left\{\sum_{\substack{w_i\neq 0\\\mathbf{w}\in \mathbb{N}_0^N}}f_{\theta}^{(k)}[\mathbf{w}];\; \theta\in\mathcal{I},\,\mathbf{M}_{\theta}(k,i)>0\right\}>\alpha.
	\end{equation} 
\end{definition}
\begin{remark}\label{x56}
	Even though we call this property uniform allowability, this property is stronger than an actual uniform allowability condition would be. Namely, for each $\theta, k, i$ such that $\mathbf{M}_{\theta}(k,i)>0$ it holds that
	\begin{equation}
		\mathbf{M}_{\theta}(k,i)=\frac{\partial f^{(k)}_{\theta}}{\partial s_i}(\mathbf{1})=
		\sum_{\mathbf{w}\in \mathbb{N} _{0}^{N }}
		f _{\theta }^{(k) }[\mathbf{w}]
		z_i
		\mathbf{1}^{\mathbf{w}}=\sum_{\substack{w_i\neq 0\\\mathbf{w}\in \mathbb{N}_0^N}}
		f _{\theta }^{(k) }[\mathbf{w}]
		z_i
		\mathbf{1}^{\mathbf{w}}>\alpha.
	\end{equation}
\end{remark}
\begin{remark}
	It is easy to see that if $\mathcal{I}$ is finite and the corresponding expectation matrices are good (in particular allowable), then the MBPRE is uniformly allowable.
\end{remark}

\begin{theorem}\label{z68}
	Let $\nu$ be an ergodic measure on $(\Sigma, \sigma)$.
	Consider the $N$-type MBPRE $\mathcal{Z}=\{\mathbf{Z}_n\}_{n\in \mathbb{N}}$ as it was defined in Definition \ref{z70}. We assume that
	\begin{enumerate}[label=\emph{(\alph*)}]
		\item $\mathcal{M}=\{\mathbf{M}_i\}_{i\in\mathcal{I}}$ (defined in \eqref{eq:z71}) is good (see Definition \ref{z19}) with respect to $\nu $.\label{z16}
		\item $\mathcal{Z}$ is uniformly allowable (\eqref{z15}). \label{z11}
		\item There exists an $M<\infty$ such that for every $\theta \in \mathcal{I}$ and $i,j\in [N]$
		      \begin{equation}
			      \frac{\partial^2 f_{\theta}^{(k)}}{\partial s_j \partial s_i}(\mathbf{1})<M.
		      \end{equation}\label{z14}
	\end{enumerate}
	Then we have that
	\begin{enumerate}
		\item if $\lambda(\nu, \mathcal{M})>0$, then
		      \begin{equation*}
			      \mathbf{q}(\overline{\pmb{\theta}})=(q^{(0)}(\overline{\pmb{\theta}}), \dots, q^{(N-1)}(\overline{\pmb{\theta}})) \ne \mathbf{1}\text{ for }\nu \text{-almost every }\overline{\pmb{\theta}}\in \Sigma.
		      \end{equation*}
		\item $\mathbf{q}(\overline{\pmb{\theta}})\ne \mathbf{1}$ for $\nu$-almost every $\overline{\pmb{\theta}}\in \Sigma$ implies that $\mathbf{q}(\overline{\pmb{\theta}})< \mathbf{1}$ for $\nu$-almost every $\overline{\pmb{\theta}}\in \Sigma$.
	\end{enumerate}
\end{theorem}
The following Corollary is an immediate consequence of the theorem.
\begin{corollary}\label{x46}
	Under the conditions of Theorem \ref{z68} if $\lambda(\nu, \mathcal{M})>0$, then $\mathbf{q}(\overline{\pmb{\theta}})< \mathbf{1}\text{ for }\nu \text{-almost every }\overline{\pmb{\theta}}\in \Sigma$.
\end{corollary}

\begin{definition}
	We say that an MBPRE $\{\mathbf{Z}_n\}_{n=0}^{\infty}$ is \texttt{strongly regular} if there exists an $n$ such that $\mathbb{P}(\min_{i\in[N]}\mathbb{P}(\|\mathbf{Z}_n\|>1|\mathbf{Z}_0=\mathbf{e}_i, \mathcal{V}=\overline{\pmb{\theta}})>0)>0$.
\end{definition}

\begin{fact}\label{x44}
	Assume that the pgfs satisfies that for a $\theta\in\mathcal{I}$ such that $\nu([\theta])>0$, we have that for all $i\in[N]$ 
	\begin{equation}
		f^{(i)}_{\theta}[\mathbf{0}]+\sum_{j\in[N]}f^{(i)}_{\theta}[\mathbf{e}_j]<1,
	\end{equation}
then the MBPRE $\{\mathbf{Z}_n\}_{n=0}^{\infty}$ is strongly regular.
\end{fact}
\begin{proof}
	The assumptions of the definitions are satisfied for $n=1$. Namely, recall that $f^{(i)}_{\theta_1}[\mathbf{0}]+\sum_{j\in[N]}f^{(i)}_{\theta_1}[\mathbf{e}_j]$ is the probability that the process starting with one individual of type $i$ conditioned on the environments first letter being $\theta_1$ has at most one individual at level 1.
	From the assumptions of the fact it follows, that for almost every $\overline{\pmb{\theta}}=(\theta, \theta_2, \dots)$
	\begin{equation}
		\min_{i\in[N]}\mathbb{P}(\|\mathbf{Z}_1\|>1|\mathbf{Z}_0=\mathbf{e}_i, \mathcal{V}=\overline{\pmb{\theta}})=1-\max_{i\in[N]}\left(f^{(i)}_{\theta_1}[\mathbf{0}]+\sum_{j\in[N]}f^{(i)}_{\theta_1}[\mathbf{e}_j]\right)>0,
	\end{equation}
	from which strong regularity follows.
\end{proof}

That is $\{\mathbf{Z}_n\}_{n=0}^{\infty}$ is \underline{not} strongly regular if and only if for every $n$, for $\nu$-almost every $\overline{\pmb{\theta}}$ there exists $i\in[N]$ such that
$$
	\mathbb{P}(\|\mathbf{Z}_n\|\leq 1|\mathbf{Z}_0=\mathbf{e}_i, \mathcal{V}=\overline{\pmb{\theta}})=1.
$$
\begin{corollary}\label{x45}
	Under the conditions of Theorem \ref{z68}
	\begin{enumerate}
		\item If $\lambda(\nu, \mathcal{M})<0$, then $\mathbf{q}(\overline{\pmb{\theta}})=\mathbf{1}\text{ for }\nu \text{-almost every }\overline{\pmb{\theta}}\in \Sigma$.
		\item If $\lambda(\nu, \mathcal{M})=0$, then either $\mathcal{Z}$ is not strongly regular, or $\mathbf{q}(\overline{\pmb{\theta}})=\mathbf{1}$ for $\nu \text{-almost every }\overline{\pmb{\theta}}\in \Sigma$.
	\end{enumerate}
\end{corollary}
\begin{proof}
	The assertions of the corollary follows from \cite[Theorem 9.6]{zbMATH03738673} and that Condition Q is precisely part $(2)$ of Theorem \ref{z68}.
\end{proof}

\begin{corollary}\label{x54}
	If $\lambda>0$ then
	\begin{equation}
	\lim_{n\to\infty} \frac{1}{n}\log \|\mathbf{Z}_n\|=\lambda
	\end{equation}
	almost surely conditioned on $\{\|\mathbf{Z}_n\|\not\to 0 \text{ as } n\to \infty\}$
\end{corollary}
	\begin{proof}
		The assertion follows from a combination of Theorem \ref{z68}, a result of Hennion and \cite[Theorem 9.6]{zbMATH03738673}. Namely, Corollary \ref{x54} is essentially the conclusion of \cite[Theorem 9.6, (3)]{zbMATH03738673}. We only need to verify that both of the assumptions of this theorem are satisfied. 
		One of them is called Condition Q. This holds because this is exactly the assertion of Theorem \ref{z68} part (2). The other assumption is the stability of the MBPRE (see \cite[Definition 9.5]{zbMATH03738673}). This holds, because of a result of Hennion, see Remark \ref{z52} in particular \eqref{z53}.

	\end{proof}
	\section{Proof of Theorem \ref{z68}}\label{z36}
Throughout this section we always use the notation of Theorem \ref{z68}.

\subsection{Preparation for the proof of Theorem \ref{z68} part (1)}

Let 
$$
	\lambda=\lim\limits_{n\to \infty}\frac{1}{n}\log \|\mathbf{M}_{\overline{\pmb{\theta}}|_n}\|,\text{ for $\nu$ almost every $\overline{\pmb{\theta}}\in\mathcal{I}^{\mathbb{N}}$}.
$$
Assume (the assumption of Theorem \ref{z68} part (1)) $\lambda>0$. Then we can choose
\begin{equation}
	\label{x21}
	0<\rho<1\quad
	\text{ such that}\quad
	1<\rho\e{\lambda}.
\end{equation}
Define the $N \times N$ matrices $\mathbf{A}_{\theta}$ for $\theta\in\mathcal{I}$ as
\begin{equation}\label{x92}
	\mathbf{A}_{\theta}=\rho \mathbf{M}_{\theta}.
\end{equation}

\begin{lemma}\label{x77}
	For $\nu$ almost every $\overline{\pmb{\theta}}\in \Sigma$ 
	\begin{equation*}
		\lim\limits_{n\to \infty}\frac{1}{n}\log (\mathbf{A}_{\overline{\pmb{\theta}}|_n})_*=\log(\rho)+\lambda>0.
	\end{equation*}
\end{lemma}
\begin{proof}
Let
\begin{equation}\label{z61}
	H:=\{\overline{\pmb{\theta}}\in \Sigma: \lim_{n\to \infty}\frac{1}{n}\log(\mathbf{M}_{\overline{\pmb{\theta}}|_n})_*\text{ exists and equals to } \lambda\}\subset\Sigma.
\end{equation}
It follows from Lemma \ref{z60} that $\nu(H)=1$.
Fix $\overline{\pmb{\theta}}=(\theta_1, \theta_2, \dots)\in H$. Then from $(\mathbf{A}_{\overline{\pmb{\theta}}|_n})_*=\rho^n(\mathbf{M}_{\overline{\pmb{\theta}}|_n})_*$ and the definition of $\rho$ the assertion follows.
\end{proof}

Define the set
\begin{equation}\label{x91}
	\mathfrak{W}:=\left\{(k,i,\theta)\in [N]^2\times\mathcal{I}: \mathbf{M}_{\theta}(k,i)>0 \right\}
	\nonumber  =\left\{(k,i,\theta)\in [N]^2\times\mathcal{I}: \mathbf{A}_{\theta}(k,i)>0 \right\}.
\end{equation}
Moreover, for $k\in [N],\ \theta \in \mathcal{I} $ we define
\begin{equation*}
	\mathfrak{W}^{\theta ,k}:=\left\{i\in [N]:
	(k,i,\theta)\in \mathfrak{W}
	\right\}.
\end{equation*}

For a matrix $\mathbf{B}$ let $\mathbf{r}_{k}(\mathbf{B})$ and $\mathbf{c}_{k}(\mathbf{B})$ denote the $k$-th row and column vector of $\mathbf{B}$ respectively.

For $\theta\in\mathcal{I}$, $\mathbf{s}\in[0,1]^N$ let 
\begin{equation}\label{x90}
	g_{\theta}^{(k)}(\mathbf{s})= 1-\mathbf{r}_{k}(\mathbf{A}_{\theta})\cdot (\mathbf{1}-\mathbf{s}), \text{ and }\mathbf{g}_{\theta}(\mathbf{s})=(g_{\theta}^{(0)}(\mathbf{s}), \dots, g_{\theta}^{(N-1)}(\mathbf{s}))\!=\!\mathbf{1}-\mathbf{A}_{\theta}(\mathbf{1}-\mathbf{s}).
\end{equation}

This immediately implies (which we frequently use without mentioning) that for $\pmb{\theta}\in\mathcal{I}^{n}$ we have for $\mathbf{s}\in[0, 1]^N$
\begin{equation}\label{x55}
\mathbf{g}_{\pmb{\theta}}(\mathbf{s})=\mathbf{1}-\mathbf{A}_{\pmb{\theta}}(\mathbf{1}-\mathbf{s}).
\end{equation}

\begin{remark}[The meaning of $\mathbf{g}_{\theta}$]\label{x74}
	For $\theta\in\mathcal{I}$, $k\in [N]$ consider $\{(\mathbf{s}, f_{\theta}^{(k)}(\mathbf{s}))\in \mathbb{R}^{N+1}: \mathbf{s}\in[0,1]^N\}$ the graph of the function $f_{\theta}^{(k)}$. We denote the tangent plane of this graph at $\mathbf{1}\in\mathbb{R}^{N}$ by
\begin{equation}\label{x62}
	t_{\theta}^{(k)}(\mathbf{s}):=f_{\theta}^{(k)}(\mathbf{1})-(f_{\theta}^{(k)})'(\mathbf{1})\cdot (\mathbf{1}-\mathbf{s})= 1-\mathbf{r}_{k}(\mathbf{M}_{\theta})\cdot (\mathbf{1}-\mathbf{s}).
\end{equation}
By Taylor's Theorem
\begin{align*}
	f^{(k)}_{\theta}(\mathbf{s}) & =t^{(k)}_{\theta}(\mathbf{s})+\frac{1}{2}(\mathbf{1}-\mathbf{s})^{T}\left(f_{\theta}^{(k)}\right)''(\mathbf{w})(\mathbf{1}-\mathbf{s}),
\end{align*}
for some $\mathbf{w} \in \{\mathbf{s}+t\left(\mathbf{1}-\mathbf{s}\right), t\in(0,1)\}$ the line segment connecting $\mathbf{s}$ and $\mathbf{1}$.

Hence, $g_{\theta}^{(k)}(\mathbf{s})$ the analogue of $t_{\theta}^{(k)}(\mathbf{s})$ using the matrices $\mathbf{A}_{\theta}$ instead of $\mathbf{M}_{\theta}$. For a visual depiction in the $N=1$ case see Figure \ref{x64}.
\end{remark}

Let
\begin{equation}\label{x99}
	B_{\delta}:=\{\mathbf{s}\in [0,1]^N: \|\mathbf{1}-\mathbf{s}\|_{\infty}\leq \delta\}
\end{equation}

\begin{lemma}\label{x76}
There exists a $\delta>0$ such that for all $\theta\in\mathcal{I}$ and $\mathbf{s}\in B_{\delta}$, $\mathbf{g}_{\theta}(\mathbf{s})\geq \mathbf{f}_{\theta}(\mathbf{s})$.

\end{lemma}
\begin{proof}[Proof of Lemma \ref{x76}]
	Recall from Remark \ref{x74}, that for $\theta\in\mathcal{I}$ and $k\in [N]$, 
\begin{align*}
	& f^{(k)}_{\theta}(\mathbf{s})=t^{(k)}_{\theta}(\mathbf{s})+\frac{1}{2}(\mathbf{1}-\mathbf{s})^{T}\left(f_{\theta}^{(k)}\right)''(\mathbf{w})(\mathbf{1}-\mathbf{s}), \\
	& g^{(k)}_{\theta}(\mathbf{s})=t^{(k)}_{\theta}(\mathbf{s})+\mathbf{r}_k
   (\mathbf{M}_{\theta}-\mathbf{A}_{\theta})(\mathbf{1}-\mathbf{s}), 
\end{align*}
where $t_{\theta}^{(k)}(\mathbf{s})=1-\mathbf{r}_{k}(\mathbf{M}_{\theta})\cdot (\mathbf{1}-\mathbf{s})$, as we stated in \eqref{x62}.

Hence, we only have to prove that there exists a $\delta>0$ such that for all $\theta\in\mathcal{I}$ and $k\in [N]$ and $\mathbf{s}\in B _{\delta}$
	\begin{equation*} 
		(\mathbf{1}-\mathbf{s})^{T}\left(f_{\theta}^{(k)}\right)''(\mathbf{w})(\mathbf{1}-\mathbf{s})\leq \mathbf{r}_k
		(\mathbf{M}_{\theta}-\mathbf{A}_{\theta})(\mathbf{1}-\mathbf{s}).
	\end{equation*}
	It follows from Fact \ref{x75} in Appendix \ref{x73}, that
	\begin{align*}
		 & (\mathbf{1}-\mathbf{s})^{T}\left(f_{\theta}^{(k)}\right)''(\mathbf{w})(\mathbf{1}-\mathbf{s}) =\sum_{i\in \mathfrak{W}^{\theta ,k}}\left(\sum_{j\in \mathfrak{W}^{\theta ,k}}(\mathbf{1}-\mathbf{s})_j \frac{\partial^2 f_{\theta}^{(k)}}{\partial s_j \partial s_i}(\mathbf{w})\right)(\mathbf{1}-\mathbf{s})_i.
	\end{align*}
	Clearly $\mathbf{r}_k
	(\mathbf{M}_{\theta}-\mathbf{A}_{\theta})(\mathbf{1}-\mathbf{s})=\sum_{i\in \mathfrak{W}^{\theta ,k}}\mathbf{r}_k(\mathbf{M}_{\theta}-\mathbf{A}_{\theta})_i(\mathbf{1}-\mathbf{s})_i.$
	Choose
	\begin{equation}\label{x69}
		\delta = \frac{(1-\rho)\alpha}{2\cdot N\cdot M}.
	\end{equation}
	Let $\mathbf{s}\in B_{\delta}$. With this choice for all $\theta\in\mathcal{I}, k\in [N]$ and $i, j\in \mathfrak{W}^{ \theta ,k}$ we have that $0\leq 1-s_j\leq \delta$. By definition, we have
	$\frac{\partial^2 f_{\theta}^{(k)}}{\partial s_j \partial s_i}(\mathbf{w}) \leq
		\frac{\partial^2 f_{\theta}^{(k)}}{\partial s_j \partial s_i}(\mathbf{1})$ for all $\mathbf{w}\in [0,1]^N$, hence
	\begin{equation*}
		\sum_{j\in \mathfrak{W}^{\theta ,k}}(\mathbf{1}-\mathbf{s})_j \frac{\partial^2 f_{\theta}^{(k)}}{\partial s_j \partial s_i}(\mathbf{w}) \leq \delta \cdot N\cdot\max_{j}
		\frac{\partial^2 f_{\theta}^{(k)}}{\partial s_j \partial s_i}(\mathbf{1})<
		\mathbf{r}_k(\mathbf{M}_\theta -\mathbf{A}_\theta )_i.
	\end{equation*}
	The second inequality holds since $\frac{\partial^2 f_{\theta}^{(k)}}{\partial s_j \partial s_i}(\mathbf{1})<M$ by assumption \ref{z14} of Theorem \ref{z68}. Whereas by the choice of the matrix $\mathbf{A}_{\theta}$ (see \eqref{x92}), 
	$\mathbf{r}_k(\mathbf{M}_\theta -\mathbf{A}_\theta )_i=(1-\rho)\mathbf{M}_\theta(k,i)$ then the uniform allowability assumption (see Remark \ref{x56}) guarantees that $\mathbf{M}_\theta(k,i)>\alpha$, hence we get that $\mathbf{r}_k(\mathbf{M}_\theta -\mathbf{A}_\theta )_i > (1-\rho)\alpha$ for $i\in \mathfrak{W}^{ \theta ,k}$.

\end{proof}
Fix the value of $\delta$ such that the assertion of Lemma \ref{x76} holds.

Now we define $\psi:[0,1]^N\to [0,1]^N$ such that whenever $\mathbf{s}\in B_{\delta}$ then $\psi(\mathbf{s})=\mathbf{s}$ but when $\mathbf{s}\notin B_{\delta}$ then $\psi(\mathbf{s})\in B_\delta$.
Namely,
\begin{equation}
	\label{z34}
	\left(\psi (\mathbf{s})\right)_j:=
	\left\{
	\begin{array}{ll}
		s_j
		, &
		\hbox{if $s_j \geq 1-\delta $;}
		\\
		1-\delta
		, &
		\hbox{if $s_j<1-\delta $}
	\end{array}
	\right.
	\quad
	\text{ for a } j\in[N].
\end{equation}

It immediately follows from the definition, that the function $\psi$ 
has the following monotonicity properties.
\begin{fact}\label{x38}
	For all $\theta \in \mathcal{I}$, $k\in [N]$
	and $\mathbf{s},\mathbf{t}\in[0,1]^N$
	we have
	\begin{enumerate}
		[label=\emph{(\alph*)}]
		\item $\mathbf{s} \leq \psi (\mathbf{s})$, and
		\item for $\mathbf{s} \leq \mathbf{t}$,\
		      $\psi (\mathbf{s}) \leq \psi (\mathbf{t})$.
	\end{enumerate}
\end{fact}

We now define for all $k\in[N], \theta\in\mathcal{I}$ and $\mathbf{s}\in [0,1]^N$
\begin{equation}\label{x97}
	h_{\theta}^{(k)}(\mathbf{s}):= g_{\theta}^{(k)}(\psi(\mathbf{s})) \text{ and }
	\mathbf{h}_{\theta}(\mathbf{s}):= (h_{\theta}^{(0)}(\mathbf{s}), \dots, h_{\theta}^{(N-1)}(\mathbf{s})).
\end{equation}

We summarize the important properties of
$\mathbf{h}_\theta (\mathbf{s})$, but before that we introduce some more notation.
Let
\begin{equation*}
	u:=
	\min\left\{
	1,\min\left\{  \mathbf{A}_{\theta}(k, i):(k, i ,\theta)\in \mathfrak{W}\right\}
	\right\}.
\end{equation*}
It follows from the uniform allowability condition that $u>0$.
Further, let $R(t)$ denote the open ball with respect to the $1$-norm centred at origin with radius $t$ in $\mathbb{R}^N$ and $R^C(t)$ its complement, namely
\begin{equation}\label{x87}
	R(t):=\{ \mathbf{s} \in [0,1]^{N}:\|\mathbf{s}\|<t\} \text{ and } R^{C}(t):=\{ \mathbf{s} \in [0,1]^{N}:\|\mathbf{s}\|\geq t\}.
\end{equation}

\begin{fact}\label{z38}
	For any $\theta \in \mathcal{I}$, $k\in [N]$ and $\mathbf{s}\leq \mathbf{t}\in[0,1]^N$ the following holds.
	\begin{enumerate}[label=\emph{(\alph*)}]
		\item If $\mathbf{s}\in B_{\delta}$,
		we have $\mathbf{h}_{\theta}(\mathbf{s})=\mathbf{g}_{\theta}(\mathbf{s})$; \label{z88}

		\item $\mathbf{h}_{\theta}(\mathbf{1})=\mathbf{1}$;\label{z90}

		\item  $\mathbf{h}_{\theta}(\mathbf{s})\leq \mathbf{h}_{\theta}(\mathbf{t})$; \label{x68}
		 
		\item For any $n\in \N$, $\pmb{\theta}\in\mathcal{I}^n$, we have that $\mathbf{h}_{\pmb{\theta}}(\mathbf{s})\geq \mathbf{f}_{\pmb{\theta}}(\mathbf{s})$;\label{z89}
		 
		\item $\mathbf{0} \leq \mathbf{h}_{\theta}(\mathbf{s})$; \label{x65}

		\item For any $v>N-u\delta$ if $\mathbf{h}_{\theta}(\mathbf{s})\in R^C(v)$, then $\mathbf{s}\in B_{\delta}$ and in particular $\mathbf{h}_{\theta}(\mathbf{s})=\mathbf{g}_{\theta}(\mathbf{s})$.\label{z87}
\end{enumerate}

\end{fact}

\begin{proof}
	\ref{z88} follows from the definition of $\mathbf{h}_{\theta}$, and \ref{z90} immediately follows from \ref{z88} combined with the fact that $\mathbf{1}\in B_{\delta}$. Part \ref{x68} is inherited from the monotonicity properties (see Lemma \ref{x38}) of $\psi$ and $\mathbf{g}_{\theta}$. \ref{x65} follows from \ref{z89}, since $\mathbf{f}_{\theta}(\mathbf{s})\geq 0$.

	\ref{z89} We use induction on $n$. First if $n=1$, then for $\mathbf{s}\in B_{\delta}$, $\mathbf{h}_{\theta}(\mathbf{s})=\mathbf{g}_{\theta}(\mathbf{s})\geq \mathbf{f}_{\theta}(\mathbf{s})$ by the definition of $\mathbf{h}_{\theta}$ and $\delta$. For $\mathbf{s}\notin B_{\delta}$, $\mathbf{h}_{\theta}(\mathbf{s})=\mathbf{g}_{\theta}(\psi(\mathbf{s}))$ on the one hand $\mathbf{g}_{\theta}(\psi(\mathbf{s})) \geq \mathbf{f}_{\theta}(\psi(\mathbf{s})) \geq \mathbf{f}_{\theta}(\mathbf{s})$ and on the other hand $\mathbf{g}_{\theta}(\psi(\mathbf{s})) \geq \mathbf{g}_{\theta}(\mathbf{s})$. Here we used that $\mathbf{g}_{\theta}$ and $\mathbf{f}_{\theta}$ are monotone increasing and that $\psi(\mathbf{s})\geq \mathbf{s}$. Now assume that $\pmb{\theta}=(\theta_1, \dots, \theta_n)\in\mathcal{I}^n$ and we know that the assumption holds for $\pmb{\theta}^-=(\theta_1, \dots, \theta_{n-1})$, then from the hypothesis and the monotonicity of $\mathbf{f}_{\theta}$, 
	$\mathbf{h}_{\pmb{\theta}}(\mathbf{s})=\mathbf{h}_{\pmb{\theta}^-}(\mathbf{h}_{\theta_n}(\mathbf{s})) \geq \mathbf{f}_{\pmb{\theta}^-}(\mathbf{h}_{\theta_n}(\mathbf{s})) \geq \mathbf{f}_{\pmb{\theta}^-}(\mathbf{f}_{\theta_n}(\mathbf{s}))=\mathbf{f}_{\pmb{\theta}}(\mathbf{s})$.

	\ref{z87} Assume $\mathbf{s}\notin B_{\delta}$, then $\|\mathbf{1}-\mathbf{s}\|_{\infty}>\delta$, i.e. there exist a $j^*\in[N]$ such that $1-s_{j^{*}}>\delta$, by the definition of $\psi$, $(\psi(\mathbf{s}))_{j^{*}}=1-\delta$. Since $\mathbf{A}_{\theta}$ is allowable, there exists a $k^{*}$ such that $\mathbf{A}_{\theta}(k^*, j^*)>0$, in particular $\mathbf{A}_{\theta}(k^*, j^*)\geq u$. Fix an arbitrary $v>N-u\delta$. It follows, that

	\begin{align*}
		v & \leq \|\mathbf{h}_{\theta}(\mathbf{s})\|=\sum_{k\in[N]} h_{\theta}^{(k)}(\mathbf{s})=\sum_{k\in[N]} g_{\theta}^{(k)}(\psi(\mathbf{s}))= \sum_{k\in[N]} 1- \sum_{j\in \mathfrak{W}^{\theta ,k}} \mathbf{A}_{\theta}(k, j)(1-\psi(\mathbf{s})_j)\\
		& \leq N- \mathbf{A}_{\theta}(k^{*}, j^{*})(1-(\psi(\mathbf{s}))_{j^{*}})\leq N-u\delta<v,
	\end{align*}
	which is contradiction.
\end{proof}

\subsection{Proof of Theorem  \ref{z68}, Part (1)}
Now we are ready to prove the first part of our main theorem.
\begin{proof}[Proof of Theorem \ref{z68}, Part (1)]

	From Lemma \ref{x77} it follows that there exists a set $H\subset \Sigma$, with $\nu(H)=1$, such that for every $\overline{\pmb{\theta}}$ there exists a $\gamma>1$ and an $\widetilde{N}=\widetilde{N}(\pmb{\theta} )$ such that for $n>\widetilde{N}$
	\begin{equation}\label{z59}
		(\mathbf{A}_{\overline{\pmb{\theta}}|_n})_*\geq \gamma^n.
	\end{equation}
	We fix such a $\gamma>1$ and $\widetilde{N}$.

	In what follows we will show that 
	\begin{equation}\label{x63}
	\text{for all }\overline{\pmb{\theta}}\in H, \quad 
	\|\mathbf{q}(\overline{\pmb{\theta}})\|=\lim_{n\to \infty}\|\mathbf{f}_{\overline{\pmb{\theta}}|_n}(\mathbf{0})\|<N, 
    \end{equation}

proving that $\mathbf{q}(\overline{\pmb{\theta}})\neq \mathbf{1}$ for $\nu$ almost every $\overline{\pmb{\theta}}$.

	\begin{figure}
		\centering
		\begin{subfigure}[b]{0.30\textwidth}
			\centering
			\includegraphics[width=\linewidth]{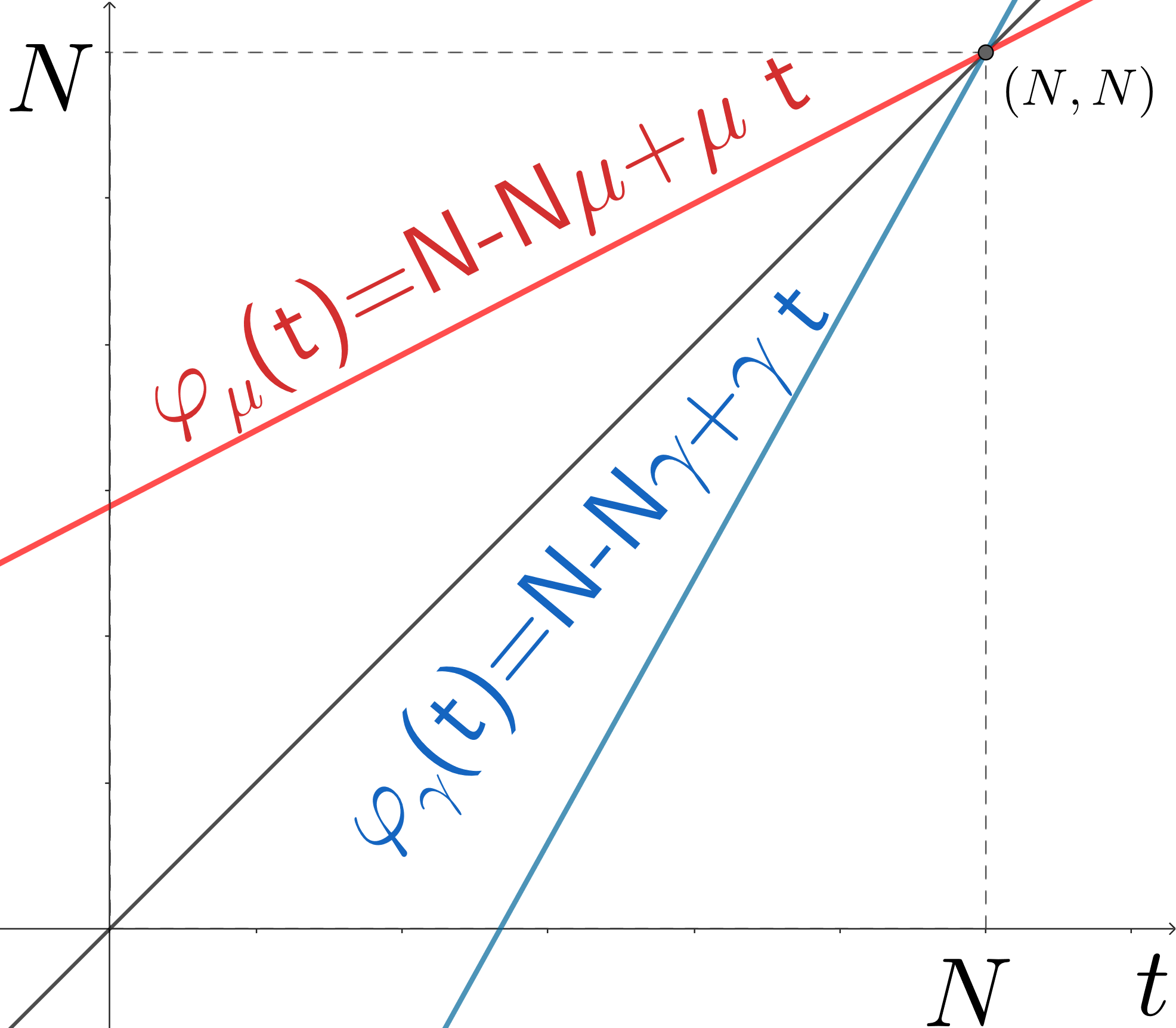}
			\caption{}\label{x61}
		\end{subfigure}
		\begin{subfigure}[b]{0.30\textwidth}
			\includegraphics[width=\linewidth]{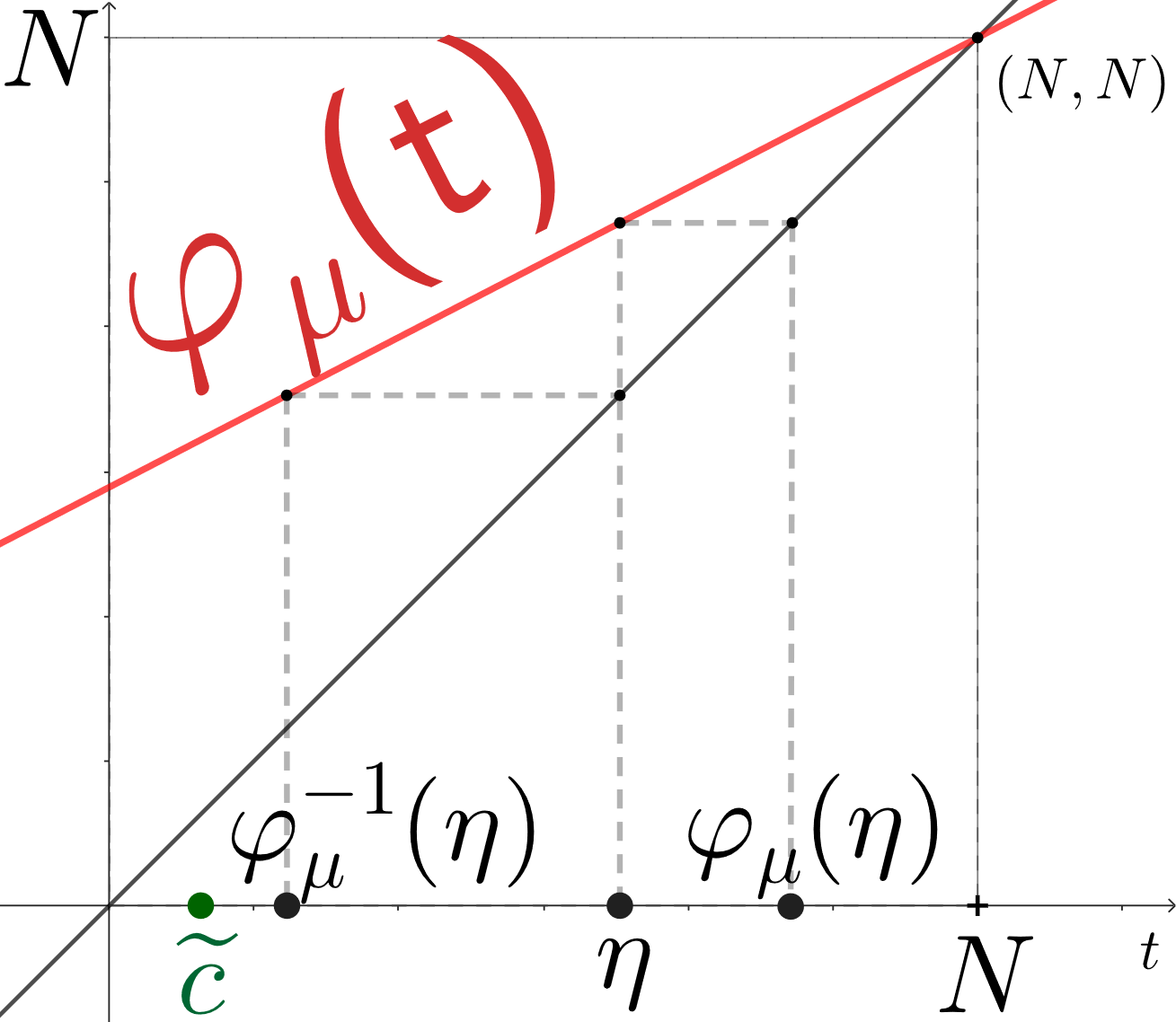}
			\caption{}\label{x60}
		\end{subfigure}
		\begin{subfigure}[b]{0.33\textwidth}
			\centering
			\includegraphics[width=\linewidth]{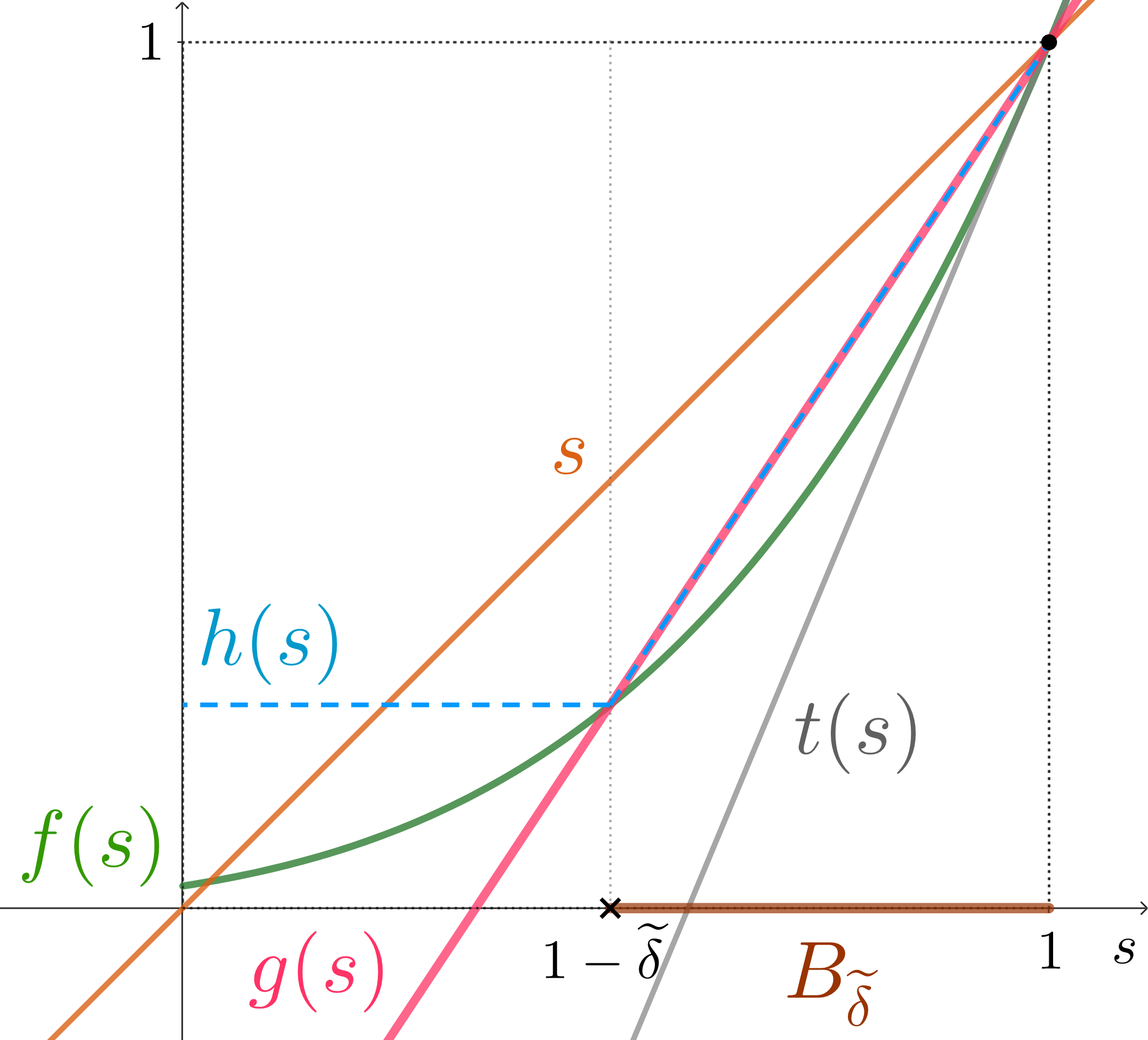}
			\caption{}\label{x64}
		\end{subfigure}
		\caption{}
	\end{figure}

	 Instead of inspecting the behaviour of $\mathbf{f}_{\overline{\pmb{\theta}}|_n}(\mathbf{0})$ directly, we consider $\mathbf{h}_{\overline{\pmb{\theta}}|_n}(\mathbf{0})\geq \mathbf{f}_{\overline{\pmb{\theta}}|_n}(\mathbf{0})$ for $n>\widetilde{N}$.
	By the uniform allowability condition we can choose $0<\mu<1$ such that
	\begin{equation}
		\mu \leq \min\limits_{k\in[N], \theta\in \mathcal{I}}
		\| \mathbf{c_{k}}(\mathbf{A}_\theta)\|.
	\end{equation}

	For the visual explanation of the following part see Figure \ref{x61}, \ref{x60}. For $v, t \in \mathbb{R}$ define
	\begin{equation}\label{x88}
		\varphi_v (t):=N-N v+v t,
	\end{equation}

	Choose $\eta<N$ be such that
	\begin{equation}
		\label{w83}
		\varphi^{-1}_\mu (\eta)>\widetilde{c}:= N-u\delta,
	\end{equation}
	where the value of 
	$\widetilde{c}$ occurred in Lemma \ref{z38} part \ref{z87}. Since $\mu<1$ there exists an $\varepsilon>0$, such that
	\begin{equation}\label{q96}
		\varphi_{\mu}(\eta)<
		\varphi_{\mu}^{\widetilde{N}}(\eta)=:
		N-\varepsilon.
	\end{equation}
	Now we fix an $m>\widetilde{N}$. Then there are two cases,
	\begin{enumerate}
		\item either $\|\mathbf{h}_{\overline{\pmb {\theta}}|_m}(\mathbf{0})\|\leq \varphi_{\mu}({\eta})$, or \label{q98}
		\item $\|\mathbf{h}_{\overline{\pmb{\theta}}|_m}(\mathbf{0})\|>\varphi_{\mu}(\eta)$. \label{q97}
	\end{enumerate}
	In case (\ref{q98}), since for all $k\in [N]$ we have $f_{\theta}^{(k)}(\mathbf{s})\leq h^{(k)}_{\theta}(\mathbf{s})$ for any $\mathbf{s}\in [0, 1]^N$ (by part \ref{z89} of Fact \ref{z38})
	and $\varphi_{\mu}(\eta)< N-\varepsilon$  (see \eqref{q96}) we get that
	\begin{equation}\label{x96}
		\|\mathbf{q}_m(\overline{\pmb{\theta}})\|=\|\mathbf{f}_{\overline{\pmb{\theta}}|_m}(\mathbf{0})\|\leq
		\varphi_{\mu} (\eta )
		<N-\varepsilon.
	\end{equation}

	In the rest of this section we consider case (\ref{q97}), namely when
	\begin{equation}\label{z26}
		\|\mathbf{h}_{\overline{\pmb{\theta}}|_m}(\mathbf{0})\| >\varphi_{\mu}(\eta).
	\end{equation}
	Set $\pmb{\theta}:=\overline{\pmb{\theta}}|_m$. In the rest of this section we always assume that $\mathbf{s}\in \left[0,1\right]^N$.
	\begin{lemma}\label{z78}
		\begin{enumerate}[label=(\roman*)]
			\item For all $\theta\in\mathcal{I},k\in[N]$, if $g^{(k)}_{\theta}(\mathbf{s})\geq 0$, then
			      \begin{equation*}
				      \|\mathbf{g}_{\theta}(\mathbf{s})\| \leq \varphi_{\mu}(\|s\|).
			      \end{equation*}
			\item For $n>\widetilde{N}$ and $\widehat{\pmb{\theta}}=(\theta_1, \dots, \theta_n)\in \mathcal{I}^n$, if for all $k\in [N]$, $g^{(k)}_{\widehat{\pmb{\theta}}}(\mathbf{s})\geq 0$, then
			      \begin{equation*}
				      \|\mathbf{g}_{\widehat{\pmb{\theta}}}(\mathbf{s})\|\leq \varphi_{\gamma}(\|s\|).
			      \end{equation*}
		\end{enumerate}
	\end{lemma}
	\begin{proof}
		We only present the proof of the first statement since the second one can be proven using the same steps and the fact (see \eqref{z59}) that $(\mathbf{A}_{\widehat{\pmb{\theta}}})_*\geq \gamma^n>\gamma$.

		\begin{align*}
			\|\mathbf{g}_{\theta}(\mathbf{s})\|
			 & =\|\mathbf{1}-\mathbf{A}_{\theta}(1-\mathbf{s})\|
			=\sum_{k=0}^{N-1}(1-\sum_{\ell=0}^{N-1} \mathbf{A}_{\theta}(k, \ell)(1-s_{\ell}))                                                                                               \\
			 &
			=N-\sum_{\ell=0}^{N-1}\sum_{k=0}^{N-1} \mathbf{A}_{\theta}(k, \ell)(1-s_{\ell})                   =  N-\sum_{\ell=0}^{N-1} \|\mathbf{c}_{\ell}(\mathbf{A}_\theta)\|(1-s_{\ell}) \\
			 &
			\leq N-(\mathbf{A}_{\theta})_*N+(\mathbf{A}_{\theta})_*\|\mathbf{s}\|
			=N-(\mathbf{A}_{\theta})_*(N-\|\mathbf{s}\|)                                                                                                                                    \\
			 & \leq N-\mu(N-\|\mathbf{s}\|)=\varphi_\mu (\|\mathbf{s}\|).
		\end{align*}
	\end{proof}
	Now we continue the proof of
	the first part of Theorem \ref{z68}.
	Define
	\begin{equation}
		T:=\left\{ p\leq m: \:\forall k\leq p,\; \mathbf{h}_{\pmb{\theta}_{k}^m}(\mathbf{0})\in R^{C}(\eta)\right\},
	\end{equation}
	where $\pmb{\theta}_{k}^m=(\theta _k,\dots ,\theta _m)$ and $m$ was fixed earlier in the proof.
	By our assumption \eqref{z26} we get that
	$\mathbf{h}_{\pmb{\theta}}(\mathbf{0})\in R^C(\varphi_\mu (\eta))\subset R^{C}(\eta)$ it follows, that $1\in T$. On the other hand, Fact \ref{z38} \ref{z87} and
	\eqref{w83} together imply
	that $m\notin T$. Namely,
	$\mathbf{h}_{\theta _m}(\mathbf{0})
		\leq
		\widetilde{c}<\varphi _{\mu }^{-1 }(\eta )<\eta $. That is $\mathbf{h}_{\theta _m}(\mathbf{0})\in R(\eta )$. So, $m\notin T$.
	Then
	\begin{align}
		1\leq Q:=\max T \leq m-1.
	\end{align}
	Let $\mathbf{v}:=\mathbf{h}_{\pmb{\theta}_{Q+1}^m}(\mathbf{0})$.
	By the definition of $Q$,
	\begin{equation}
		\label{w81}
		\mathbf{v}\in R(\eta)
	\end{equation}
	Also, for any $k\leq Q$, $\mathbf{h}_{\theta_k}(\mathbf{h}_{\pmb{\theta}_{k+1}^m}(\mathbf{0}))\in R^{C}(\eta)$, hence by $\eta > \widetilde{c}$ and Fact \ref{z38} \ref{z87} it follows that for any $k\leq Q$
	\begin{equation*}
		\mathbf{h}_{\theta_k}(\mathbf{h}_{\pmb{\theta}_{k+1}^m}(\mathbf{0}))=
		\mathbf{g}_{\theta_k}(\mathbf{h}_{\pmb{\theta}_{k+1}^m}(\mathbf{0})).
	\end{equation*}
	By repeated applications of this we get that for any $k\leq Q$,
	\begin{equation}\label{w82}
		\mathbf{h}_{\pmb{\theta}_k^Q}(\mathbf{v})=\mathbf{g}_{\pmb{\theta}_k^Q}(\mathbf{v}).
	\end{equation}
	Now we show that $Q$ can not be too big. Namely, to get a contradiction
	assume that $Q>\widetilde{N}$. Using \eqref{z26}, \eqref{w82},
	the second part of Lemma
	\ref{z78} together with the assumption $Q>\widetilde{N}$, the fact that $\gamma >1$, and \eqref{w81}, in this order, we obtain that
	\begin{equation*}
		\varphi_{\mu}(\eta)<\|\mathbf{h}_{\pmb{\theta}_{1}^{m}}(\mathbf{0})\|=\|\mathbf{g}_{\pmb{\theta}_{1}^{Q}}(\mathbf{v})\| \leq \varphi_{\gamma} (\|\mathbf{v}\|)\leq \|\mathbf{v}\|<\eta,
	\end{equation*}
	which is a contradiction since $\eta< \varphi_{\mu}(\eta)$.

	Hence, $Q\leq \widetilde{N}$. However, in this case using the first part of Lemma \ref{z78} $\widetilde{N}$-times, we get that
	\begin{multline*}
		\|\mathbf{h}_{\pmb{\theta}}(\mathbf{0})\|=\|\mathbf{g}_{\pmb{\theta}_{1}^{Q}}(\mathbf{v})\| \leq \varphi_{\mu}(\|\mathbf{g}_{\pmb{\theta}_{2}^{Q}}(\mathbf{v})\|) \leq \dots \leq \varphi_{\mu}^{Q-1}(\|
		{g}_{\pmb{\theta}_{Q}^{Q}}
		\mathbf{v}\|)
		\\
		\leq \varphi_{\mu}^{Q}(\|
		\mathbf{v}\|)
		\leq \varphi_{\mu}^{\widetilde{N}}(\eta)=N-\varepsilon,
	\end{multline*}
	where $\varepsilon>0$ was defined in \eqref{q96}.
	This means that (similarly to \eqref{x96}),
	\begin{equation*}
		\|\mathbf{q}_m(\overline{\pmb{\theta}})\|=\|\mathbf{f}_{\pmb{\theta}|_m}(\mathbf{0})\|
		\leq \|\mathbf{h}_{\pmb{\theta}}(\mathbf{0})\| \leq N-\varepsilon.
	\end{equation*}

	This finishes the treatment of case \ref{q97}. It follows that for $m >\widetilde{N}$ we have
	$\|\mathbf{q}_m(\overline{\pmb{\theta}})\| \leq N-\varepsilon$. In this way we have verified that
	\eqref{x63} holds, which completes the proof of the first part of Theorem \ref{z68}.
\end{proof}

\subsection{The proof of Theorem \ref{z68}, Part (2)}
It follows from the
assumption \ref{z16} of Theorem \ref{z68}
that there exists a finite word
$\widetilde{\pmb{\theta }}=(\widetilde{\theta} _1,\dots ,\widetilde{\theta} _p)\in \mathcal{I}^p$ such that $\nu([\widetilde{\pmb{\theta }}])>0$ and all elements of
$\mathbf{M}_{\widetilde{\pmb{\theta }}}:=
	\mathbf{M}_{\widetilde{\theta}_1}\cdots\mathbf{M}_{\widetilde{\theta}_n}$ are strictly positive. From the ergodicity of $\nu$ (Principal Assumption
\ref{z99}) it follows, that there exists $\widetilde{\Sigma}\subset \Sigma$ with $\nu(\widetilde{\Sigma})=1$, such that all $\overline{\pmb{\theta}}\in\widetilde{\Sigma}$ contains $\widetilde{\pmb{\theta}}$ as a subword.

For a $k_1\in[N]$ we define
\begin{equation}\label{x89}
	\text{\bf Bad}_{k_1}:=
	\left\{
	\overline{\pmb{\theta}}\in \widetilde{\Sigma }
	:
	\lim\limits_{n\to\infty}
	f _{\overline{\pmb{\theta}}|_n }^{(k_1) }(\mathbf{0})=1
	\right\}.
\end{equation}

\begin{lemma}\label{x10}
	For any $k_1\in[N]$, 		$\nu (\text{\bf Bad}_{k_1})=0$.
\end{lemma}
\begin{proof}
	Fix an $n \geq p$. Let
	\begin{equation*}
		A_n:=\left\{\overline{\pmb{\theta}}\in \widetilde{\Sigma} :
		\overline{\pmb{\theta}}|_n
		\text{ does \underline{not} contain the word $\widetilde{\pmb{\theta }}$ }
		\right\},\text{ then }\bigcup_{n=p}^{\infty}A_n^C=\widetilde{\Sigma}.
	\end{equation*}

	Hence, it is enough to prove that for any $n\geq p$ and any $\varepsilon$
	\begin{equation}
		\nu(\text{\bf Bad}_{k_1}\cap A_n^C)<\varepsilon.
	\end{equation}
	Let $n$ and $\varepsilon >0$ be arbitrary.
	Recall that $R^{C}(t):=\{ \mathbf{s} \in [0,1]^{N}:\|\mathbf{s}\|\geq t\}$.
	By the first part of Theorem \ref{z68}, we can choose a $\delta >0$ which is so small that for
	\begin{equation}
		\label{x06}
		t:=
		N(1-\delta p_*^{-n}),\text{ and }
		X_t:=
		\left\{\overline{\pmb{\theta}}:\lim\limits_{\ell \to\infty} f_{\overline{\pmb{\theta}}|_{\ell }}(\mathbf{0})\in R^C(t)\right\},
	\end{equation}
	we have
	\begin{equation*}
		\nu (X_t)<\varepsilon,
	\end{equation*}
	where $p_{*}=\frac{\alpha}{2}$ (recall that $\alpha$ was defined in \eqref{z15}).\begin{lemma}\label{x34}
		Let $\theta \in \mathcal{I}$, $k\in[N]$
		and $0<\widetilde{\delta}<1$.
		Then for any $\mathbf{s}\in[0,1]^N$
		\begin{equation*}
			\text{if there exists an } i\in \mathfrak{W}^{\theta, k} \text{ such that }s_i<1-\widetilde{\delta}
			\Longrightarrow
			f_\theta ^{(k)}(\mathbf{s})<1-p_{*}\widetilde{\delta}.
		\end{equation*}
	\end{lemma}
	\begin{proof}
		Fix an $\mathbf{s}\in[0,1]^N$ such that $s_i<1-\widetilde{\delta}$
		\begin{align*}
			  & f _{\theta }^{(k) }(\mathbf{s})
			\leq f _{\theta }^{(k) }(\mathbf{1}-\mathbf{e}_i \widetilde{\delta})
			= \sum_{\mathbf{ z}\in \mathbb{N} _{0}^{N }}
			f _{\theta }^{(k) }[\mathbf{z}](1-\widetilde{\delta})^{z_{i}}
			\leq \sum_{\substack{\mathbf{ z}\in \mathbb{N} _{0}^{N }                     \\ z_i=0}}f _{\theta }^{(k) }[\mathbf{z}]\\
			+ & \sum_{\substack{\mathbf{ z}\in \mathbb{N} _{0}^{N }                      \\z_i\neq 0}}
			f _{\theta }^{(k) }[\mathbf{z}](1-\widetilde{\delta})
			\leq 1-\widetilde{\delta}\sum_{\substack{\mathbf{ z}\in \mathbb{N} _{0}^{N } \\z_i\neq 0}}f_{\theta }^{(k) }[\mathbf{z}]
			< 1-p_{*}\widetilde{\delta}
		\end{align*}
	\end{proof}
	Now we fix a $\overline{\pmb{\theta}}\in A _{n}^{C }\bigcap \text{\bf Bad}_{k_1}$.
	For an $r\in\mathbb{N}$, $r \geq 2$ let
	\begin{align*}
		C:= & \left\{k\in[N]:
		\exists (k_2,\dots ,k_n)\in [N]^{n-1}
		\text{ such that }
		k_i\in \mathfrak{W}^{\theta _{i-1},k_{i-1}},i\in\{2,\dots,n\},
		k\in \mathfrak{W}^{\theta _{n},k_{n}}
		\right\}              \\
		=   & \left\{k\in[N]:
		\mathbf{M}_{\overline{\pmb{\theta}}|_n}(k_1,k)>0
		\right\}
	\end{align*}
	\begin{fact}\label{x03}
		$C=[N]$.
	\end{fact}
	\begin{proof}[Proof of Fact \ref{x03}]
		This follows from the fact that all the expectation matrices are allowable and that $\overline{\pmb{\theta}}|_n$ contains the word $\widetilde{\pmb{\theta}}$, which implies that $\mathbf{M}_{\overline{\pmb{\theta}}|_n}$ is a strictly positive matrix.
	\end{proof}
	Since we assumed that $\overline{\pmb{\theta}}\in \text{\bf Bad}_{k_1}$ we can find a $K>n$
	such that
	\begin{equation}
		\label{x01}
		f_{\overline{\pmb{\theta}}|_K}^{(k_1)}(\mathbf{0})
		>1-\delta.
	\end{equation}
	As earlier, we write $
		\overline{\pmb{\theta}}_{j}^{\ell }=
		(\theta _j,\theta _{j+1},\dots ,\theta _\ell )$.
	For every $i<K$ let
	\begin{equation}
		\label{w99}
		\mathbf{s}_i=
		(\mathbf{s}_i(0),\dots ,\mathbf{s}_i(N-1)):=
		\mathbf{f}_{\overline{\pmb{\theta}} _{i+1}^{K }}(\mathbf{0}).
	\end{equation}
	Then
	\begin{equation*}
		f^{(k_1)}_{\overline{\pmb{\theta}}|_K}(\mathbf{0})=
		f _{\theta _1}^{(k_1) }(\mathbf{s}_1),
		\qquad
		\mathbf{s}_i(\ell )
		=f _{\theta _{i+1}}^{(\ell ) }
		(\mathbf{s}_{i+1})
		=f _{\theta _{i+1}}^{(\ell ) }
		(
		f _{\overline{\pmb{\theta}}_{i+2}^{n }}(\mathbf{0}))
		.
	\end{equation*}
	It follows from Lemma \ref{x34}, and formulas \eqref{x01}, \eqref{w99} that
	\begin{equation*}
		\forall k_2\in \mathfrak{W}^{\theta _1,k_1},\quad
		\mathbf{s}_1(k_2)>1-\delta p_{*}^{-1}.
	\end{equation*}
	
	By repeated application of Lemma \ref{x34} we get 
	\begin{equation*}
		\mathbf{s}_n(k_{n+1})>1-\delta p_{*}^{-n},\quad
		\forall (k_2,\dots ,k_{n+1})\in C^n.
	\end{equation*}
	Using this, by Fact \ref{x03} we get that
$\mathbf{s}_n(k)>1-\delta p_{*}^{-n}$ for all $k\in[N]$.
	This means that
	\begin{equation}
		\label{w96}
		\mathbf{f}_{\overline{\pmb{\theta}}_{n+1}^{K }}(\mathbf{0})=
		\mathbf{s}_n\in B_{\delta p_{*}^{-n}}.
	\end{equation}
	Using that $\mathbf{f}_\theta $ is componentwise monotone and $\mathbf{f}_\theta :[0,1]^N\to[0,1]^N$,
	we get from \eqref{w96} that
	$$
		\lim\limits_{M\to\infty}
		\mathbf{f}_{\overline{\pmb{\theta}}_{n+1}^{M}}(\mathbf{0})\in B_{\delta p_{*}^{-n}}\subset R^C(t),
	$$
	where $t$ and $X_t$ were defined in \eqref{x06}.
	That is we have proved that
	$
		\overline{\pmb{\theta}}\in A _{n}^{C }\bigcap \text{\bf Bad}_{k_1}
		\Longrightarrow
		\sigma^{n}\overline{\pmb{\theta}}
		\in
		X_t.
	$
	In other words we have verified that
	\begin{equation*}
		A _{n}^{C }\cap \text{\bf Bad}_{k_1}\subset
		\sigma ^{-n}X_t.
	\end{equation*}
	Using that $\nu $ is measure preserving, we get that
	\begin{equation*}
		\nu (A _{n}^{C }\cap \text{\bf Bad}_{k_1})
		\leq
		\nu (\sigma ^{-n}X_t)=
		\nu (X_t)<\varepsilon.
	\end{equation*}
\end{proof}
\begin{proof}[Proof of Theorem \ref{z68}, Part (2)]
	Using that $k_1\in[N]$ was arbitrary
	in Lemma \ref{x10},
	we get that the second assertion of Theorem \ref{z68} holds.
\end{proof}

\begin{appendix}
\section{Theorem of Hennion}\label{x82}
Let $\mathbf{B}$ be an $N\times N$ non-negative and allowable ($\mathbf{B}\in\mathscr{A}$) matrix. 
It is easy to check that the norms defined in Section \ref{z49} agree with the ones used in \cite{Hennion97}, i.e. 
\begin{align*}
	\|\mathbf{B} \|_1&=
	\max\left\{
	\sum\limits_{i\in[N]}\sum\limits_{j\in[N]}\mathbf{B}_{i, j}x_j:
	x_j\geq 0,\ \sum\limits_{j\in[N]}x_j=1
	\right\},\text{ and } \\
	(\mathbf{B})_*&=\min\left\{
	\sum\limits_{i\in[N]}\sum\limits_{j\in[N]}\mathbf{B}_{i, j}x_j:
	x_j\geq 0,\ \sum\limits_{j\in[N]}x_j=1
	\right\}.
\end{align*}
Let $\mathcal{B}=\{\mathbf{B}_i\}_{i\in\mathcal{I}}\subset \mathscr{A}$. For $\pmb{\theta}=(\theta_1, \dots, \theta_n)\in\mathcal{I}^n$ we denote
\begin{equation*}
	\mathbf{B}_{\pmb{\theta}}:=\mathbf{B}_{\theta_1}\cdots \mathbf{B}_{\theta_n} \text{ and }
	\mathbf{B}_{\cev{\pmb{\theta}}}:=\mathbf{B}_{\theta_n}\cdots \mathbf{B}_{\theta_1}.
\end{equation*}
\subsection{A corollary of a theorem of Hennion}\label{z48}
The random matrix $X^{(n)}$, which appears in
\cite[Theorem 2]{Hennion97},
corresponds to
$\mathbf{B}^{T}_{\stackrel{\leftarrow}{\overline{\pmb{\theta}}|_n}}$
where $\overline{\pmb{\theta}}$ is chosen randomly according to
the probability measure $\nu $.
\cite[Theorem 2]{Hennion97} has two conditions: The first one is that $m_1<\infty$ and the second one is
Condition $\mathscr{C}$ which corresponds to the first and second part of our assumption that $\mathcal{B}$ is good (see Definition \ref{z19}).
Hence, the conditions of
\cite[Theorem 2]{Hennion97} always holds whenever
$\mathcal{B}\subset \mathscr{A}$ is good with respect to the ergodic measure $\nu $. The conclusion of \cite[Theorem 2]{Hennion97} immediately implies that
\begin{equation}\label{z39}
	\lim\limits_{n\to \infty} \sup_{i\in [N]}\left|\frac{1}{n}\log \left(\mathbf{1}^T \mathbf{B}_{\overline{\pmb{\theta}}|_n}\mathbf{e}_i\right)-\lambda\right|=0, \text{ for } \nu \text{-a.e. }\overline{\pmb{\theta}}\in \Sigma.
\end{equation}
Observe that $\mathbf{1}^T \mathbf{B}_{\overline{\pmb{\theta}}|_n}\mathbf{e}_i$ is the $i$-th column sum of the matrix
$\mathbf{B}_{\overline{\pmb{\theta}}|_n}$. Hence, for $\nu$-almost every $\overline{\pmb{\theta}}\in\Sigma$
	\begin{equation}
		\lim\limits_{n\to \infty} \left|\frac{1}{n}\log(\mathbf{B}_{\overline{\pmb{\theta}}|_n})_{*}-\lambda\right|=0.
	\end{equation}
	\begin{remark}\label{z52}
		Observe that in \cite[Theorem 2]{Hennion97} the role of $\mathbf{1}^T$ and $\mathbf{e}_i$ is interchangeable, meaning that we can consider row, instead of column sums of the matrices. Consequently, we get an analogous result to Lemma \ref{z60} but for the ``row-sum exponent'', namely that
		\begin{equation}\label{x53}
		\lim_{n\to \infty}\frac{1}{n}\log(\min_{i\in[N]}\mathbf{e}_i^T \mathbf{B}_{\overline{\pmb{\theta}}|n}\mathbf{1})=\lambda \quad \text{ for $\nu$-a.e. }\overline{\pmb{\theta}}\in \Sigma.
		\end{equation}
	\end{remark}

\section{Basic properties of multivariate pgfs}\label{x73}

The following is a well-known fact.
\begin{fact}\label{x79}
	Let $X,Y$ be independent random variables on $(\Omega, \mathcal{F}, \mathbb{P})$ taking values in $\mathbb{N}_0^N$ with pgfs $f_X$ and $f_Y$ respectively. Then the pgf $f_{X+Y}$ of the random variable $X+Y$ satisfies $f_{X+Y}(\mathbf{s})=f_X(\mathbf{s})\cdot f_Y(\mathbf{s})$.
\end{fact}
\begin{fact}\label{x75}
	If $(k, i, \theta)\notin \mathfrak{W}$ (or equivalently $\partial f^{(k)}_{\theta}/\partial s_i(\mathbf{1})=0$) then for all 
	$\mathbf{s}\in[0,1]^N$
	\begin{enumerate}
		\item $\partial f^{(k)}_{\theta}/ \partial s_i(\mathbf{w})=0$ and \label{x72}
		\item $\mathbf{r}_i((f_{\theta}^{(k)})''(\mathbf{w}))$=$\mathbf{c}_i((f_{\theta}^{(k)})''(\mathbf{w}))=\mathbf{0}$. \label{x71}
	\end{enumerate}
\end{fact}
\begin{proof}
	Since \ref{x71} immediately follows from \ref{x72}, we only provide details for the first part.
	By definition 
	\begin{equation*}
		f _{\theta }^{(k) }(\mathbf{s}):=
		\sum_{\mathbf{ z}\in \mathbb{N} _{0}^{N }}
		f _{\theta }^{(k) }[\mathbf{z}]\mathbf{s}^{\mathbf{z}},
		\
		\mathbf{s}\in [0,1]^N,\ k\in [N],\ \theta \in \mathcal{I},
	\end{equation*}
	hence 
	\begin{equation}\label{x51}
		\frac{\partial f^{(k)}_{\theta}}{\partial s_i}(\mathbf{s})=
		\sum_{\mathbf{ z }\in \mathbb{N} _{0}^{N }}
		f _{\theta }^{(k) }[\mathbf{z}]
		z_i s _{i}^{-1 }
		\mathbf{s}^{\mathbf{z}}.
	\end{equation}
	From 
	$$
	0=\partial f^{(k)}_{\theta}/\partial s_i(\mathbf{1})=\sum_{\mathbf{ z}\in \mathbb{N} _{0}^{N }}
	f _{\theta }^{(k) }[\mathbf{z}]z_i
\Longrightarrow
	\partial f^{(k)}_{\theta}/ \partial s_i(\mathbf{s})=0.$$
\end{proof}

\end{appendix}

\begin{acks}[Acknowledgments]
	We would like to say thanks to De-Jun Feng, who called our attention to Hennion's result \cite{Hennion97} which played a crucial role in proving Theorem \ref{z68}. We also would like to say thanks to Michel Dekking and Alex Rutar for their helpful comments and suggestions.
\end{acks}

\begin{funding}
	VO is supported by National Research,
Development and Innovation Office - NKFIH, Project FK134251.
	
KS is supported by National Research,
Development and Innovation Office - NKFIH, Project K142169.

Both authors have received funding from the HUN-REN Hungarian Research
Network.
\end{funding}

\bibliographystyle{imsart-nameyear} 
\bibliography{references}       

\begin{thebibliography}{14}

\bibitem[\protect\citeauthoryear{Athreya and Karlin}{1971}]{bp_renv}
\begin{barticle}[author]
\bauthor{\bsnm{Athreya},~\bfnm{K.~B.}\binits{K.~B.}} \AND \bauthor{\bsnm{Karlin},~\bfnm{S.}\binits{S.}}
(\byear{1971}).
\btitle{On Branching Processes with Random Environments: I: Extinction Probabilities}.
\bjournal{The Annals of Mathematical Statistics}
\bvolume{42}
\bpages{1499--1520}.
\end{barticle}
\endbibitem

\bibitem[\protect\citeauthoryear{Dekking and Grimmet}{1988}]{zbMATH04022321}
\begin{barticle}[author]
\bauthor{\bsnm{Dekking},~\bfnm{F.~M.}\binits{F.~M.}} \AND \bauthor{\bsnm{Grimmet},~\bfnm{G.~R.}\binits{G.~R.}}
(\byear{1988}).
\btitle{Superbranching processes and projections of random {Cantor} sets}.
\bjournal{Probability Theory and Related Fields}
\bvolume{78}
\bpages{335--355}.
\bdoi{10.1007/BF00334199}
\end{barticle}
\endbibitem

\bibitem[\protect\citeauthoryear{Dekking and Meester}{1990}]{dekking1990structure}
\begin{barticle}[author]
\bauthor{\bsnm{Dekking},~\bfnm{Frederik~Michel}\binits{F.~M.}} \AND \bauthor{\bsnm{Meester},~\bfnm{Ronaldus Wilhelmus~Jozef}\binits{R.~W.~J.}}
(\byear{1990}).
\btitle{On the structure of Mandelbrot's percolation process and other random Cantor sets}.
\bjournal{Journal of Statistical Physics}
\bvolume{58}
\bpages{1109--1126}.
\end{barticle}
\endbibitem

\bibitem[\protect\citeauthoryear{Dekking and Simon}{2008}]{zbMATH05255663}
\begin{barticle}[author]
\bauthor{\bsnm{Dekking},~\bfnm{Michel}\binits{M.}} \AND \bauthor{\bsnm{Simon},~\bfnm{K{\'a}roly}\binits{K.}}
(\byear{2008}).
\btitle{On the Size of the Algebraic Difference of Two Random {Cantor} Sets}.
\bjournal{Random Structures \& Algorithms}
\bvolume{32}
\bpages{205--222}.
\bdoi{10.1002/rsa.20178}
\end{barticle}
\endbibitem

\bibitem[\protect\citeauthoryear{Falconer}{1989}]{falconer1989}
\begin{barticle}[author]
\bauthor{\bsnm{Falconer},~\bfnm{K.~J.}\binits{K.~J.}}
(\byear{1989}).
\btitle{Projections of Random {Cantor} Sets}.
\bjournal{Journal of Theoretical Probability}
\bvolume{2}
\bpages{65--70}.
\bdoi{10.1007/BF01048269}
\end{barticle}
\endbibitem

\bibitem[\protect\citeauthoryear{Falconer and Grimmett}{1992}]{zbMATH00064453}
\begin{barticle}[author]
\bauthor{\bsnm{Falconer},~\bfnm{K.~J.}\binits{K.~J.}} \AND \bauthor{\bsnm{Grimmett},~\bfnm{G.~R.}\binits{G.~R.}}
(\byear{1992}).
\btitle{On the geometry of random {Cantor} sets and fractal percolation}.
\bjournal{Journal of Theoretical Probability}
\bvolume{5}
\bpages{465--485}.
\bdoi{10.1007/BF01060430}
\end{barticle}
\endbibitem

\bibitem[\protect\citeauthoryear{Hennion}{1997}]{Hennion97}
\begin{barticle}[author]
\bauthor{\bsnm{Hennion},~\bfnm{H.}\binits{H.}}
(\byear{1997}).
\btitle{{Limit theorems for products of positive random matrices}}.
\bjournal{The Annals of Probability}
\bvolume{25}
\bpages{1545 -- 1587}.
\bdoi{10.1214/aop/1023481103}
\end{barticle}
\endbibitem

\bibitem[\protect\citeauthoryear{Kersting and Vatutin}{2017}]{Kersting2017}
\begin{binproceedings}[author]
\bauthor{\bsnm{Kersting},~\bfnm{G.}\binits{G.}} \AND \bauthor{\bsnm{Vatutin},~\bfnm{V.}\binits{V.}}
(\byear{2017}).
\btitle{Discrete Time Branching Processes in Random Environment}.
\end{binproceedings}
\endbibitem

\bibitem[\protect\citeauthoryear{M{\'o}ra, Simon and Solomyak}{2009}]{zbMATH05643268}
\begin{barticle}[author]
\bauthor{\bsnm{M{\'o}ra},~\bfnm{P{\'e}ter}\binits{P.}}, \bauthor{\bsnm{Simon},~\bfnm{K{\'a}roly}\binits{K.}} \AND \bauthor{\bsnm{Solomyak},~\bfnm{Boris}\binits{B.}}
(\byear{2009}).
\btitle{The {Lebesgue} measure of the algebraic difference of two random {Cantor} sets}.
\bjournal{Indagationes Mathematicae. New Series}
\bvolume{20}
\bpages{131--149}.
\bdoi{10.1016/S0019-3577(09)80007-4}
\end{barticle}
\endbibitem

\bibitem[\protect\citeauthoryear{Pollicott and Vytnova}{2023}]{pollicottVytnova}
\begin{bmisc}[author]
\bauthor{\bsnm{Pollicott},~\bfnm{Mark}\binits{M.}} \AND \bauthor{\bsnm{Vytnova},~\bfnm{Polina}\binits{P.}}
(\byear{2023}).
\bhowpublished{Personal communication}.
\end{bmisc}
\endbibitem

\bibitem[\protect\citeauthoryear{Simon and Orgov{\'a}nyi}{2022}]{OurPaper}
\begin{barticle}[author]
\bauthor{\bsnm{Simon},~\bfnm{K.}\binits{K.}} \AND \bauthor{\bsnm{Orgov{\'a}nyi},~\bfnm{V.}\binits{V.}}
(\byear{2022}).
\btitle{Projections of the random Menger sponge}.
\bjournal{arXiv preprint arXiv:2205.03125}.
\end{barticle}
\endbibitem

\bibitem[\protect\citeauthoryear{Tanny}{1981}]{zbMATH03738673}
\begin{barticle}[author]
\bauthor{\bsnm{Tanny},~\bfnm{David}\binits{D.}}
(\byear{1981}).
\btitle{On multitype branching processes in a random environment}.
\bjournal{Advances in Applied Probability}
\bvolume{13}
\bpages{464--497}.
\bdoi{10.2307/1426781}
\end{barticle}
\endbibitem

\bibitem[\protect\citeauthoryear{Walters}{2000}]{walters2000introduction}
\begin{bbook}[author]
\bauthor{\bsnm{Walters},~\bfnm{P.}\binits{P.}}
(\byear{2000}).
\btitle{An Introduction to Ergodic Theory}.
\bseries{Graduate Texts in Mathematics}.
\bpublisher{Springer New York}.
\end{bbook}
\endbibitem

\bibitem[\protect\citeauthoryear{Weissner}{1971}]{zbMATH03349140}
\begin{barticle}[author]
\bauthor{\bsnm{Weissner},~\bfnm{E.~W.}\binits{E.~W.}}
(\byear{1971}).
\btitle{Multitype branching processes in random environments}.
\bjournal{J. Appl. Probab.}
\bvolume{8}
\bpages{17--31}.
\bdoi{10.2307/3211834}
\end{barticle}
\endbibitem

\end{thebibliography}

\end{document}